\documentclass[a4paper,10pt]{siamltex1213}
\usepackage{graphicx}
\usepackage{amsfonts}
\usepackage{amsmath}
\usepackage{amssymb}
\usepackage{fancyhdr}
\usepackage{mathtools}
\usepackage{indentfirst}
\usepackage{placeins}
\usepackage{listings}
\usepackage{caption}
\usepackage{subfigure}
\usepackage{lipsum}
\usepackage{xcolor}
\usepackage{comment}
\usepackage{url}
\usepackage{bm}

\def\longrightharpoonup{\relbar\joinrel\rightharpoonup}
\def\longleftharpoondown{\leftharpoondown\joinrel\relbar}

\def\longrightleftharpoons{
  \mathop{
    \vcenter{
       \hbox{
       \ooalign{
          \raise1pt\hbox{$\longrightharpoonup\joinrel$}\crcr
  	  \lower1pt\hbox{$\longleftharpoondown\joinrel$}
	}
      }
    }
  }
}

\newtheorem{remark}{Remark}

\begin{document}

\title{{Asymptotic analysis of Multiscale Markov chain}}
\date{}
\author{Wei Zhang\textsuperscript{1}}
\footnotetext[1]{Institute of Mathematics, Free University of Berlin,
Arnimallee 6, 14195 Berlin, Germany. Email : wei.zhang@fu-berlin.de}
\maketitle
\begin{abstract}
    We consider continuous-time Markov chain on a finite state space $X$.
    We assume $X$ can be clustered into several subsets such that
    the intra-transition rates within these subsets
    are of order $\mathcal{O}(\frac{1}{\epsilon})$
    comparing to the inter-transition rates among them, where $0 < \epsilon \ll 1$. 
    Several asymptotic results are obtained as $\epsilon \rightarrow 0$ concerning the convergence of Kolmogorov
    backward equation, Poincar\'e constant, (modified) logarithmic Sobolev constant to
    their counterparts of certain reduced Markov chain. Both reversible and
    irreversible Markov chains are considered.
\end{abstract}

\section{Introduction}
\label{sec-intro}
\subsection{Multiscale Markov chain}
\label{subsec-setup}
In recent decades, Markov chains have been 
intensively investigated due to their effectiveness in modeling systems
arising from biology, physics, economics et al.  \cite{mc-stability, mc-mixing, norris-mc}.  
Inspired by new phenomena from these disciplines, new topics related to Markov
chains are continuously emerging and attracting researchers' attentions.
Metastability in Markov chains is one such interesting topic which tries to understand systems'
behaviors on large time scales by eliminating systems' oscillations on
short time scales and identifying certain effective dynamics on large time scales
\cite{msm_gen_valid, Deuflhard200039}.

In this work, we consider a continuous-time Markov chain $\mathcal{C}$ on finite state space $X = \{x_1, x_2, \cdots, x_n\}$.
We assume $\mathcal{C}$ is irreducible and therefore has a unique invariant
measure \cite{mc-mixing, norris-mc}. 
Suppose state space $X$ can be clustered
into $m$ ($m > 1$) nonempty disjoint subsets $X_1, X_2, \cdots, X_m$, with $|X_i| =
n_i > 0$, $\sum\limits_{i=1}^{m} n_i = n$.
We will be interested in the situation when transitions of system's states
within the same subset occur much more
frequently than transitions between states belonging to different subsets. Precisely, let
$n\times n$ matrix $Q$ be the infinitesimal
generator of Markov chain $\mathcal{C}$, which we assume can be written as
\begin{align}
 Q = \frac{1}{\epsilon}Q_0 + Q_1\,, 
 \label{q-form}
\end{align} 
 for some parameter $0 < \epsilon \ll 1$. Matrices $Q_0$ and
$Q_1$ satisfy that 
\begin{enumerate}
  \item
    $Q_0(x,y) \ge 0, \, Q_1(x,y) \ge 0$, ~if $x \neq y$.
  \item
    Each row sum of their entries equals zero, i.e.
    \begin{align*}
       \sum_{y \in X} Q_0(x,y) = \sum_{y\in X} Q_1(x,y) = 0\,, \quad \forall x \in
       X\,.
    \end{align*}
\item
$Q_0(x,y) = 0$, ~if $x \in X_i$, $y \in X_j$ for some $1 \le i \neq j\le m$.
\item
$Q_1(x,y) = 0$, ~if $x,y \in X_i$ for some $1 \le i \le m$ and $x \neq y$. 
\end{enumerate}
That is,
$Q_0$
and $Q_1$ describe the intra- and inter-transition rates among subsets $X_i$,
respectively.
Notice that comparing to Chapter~$5$ and~$9$ of \cite{pavliotis2008multiscale}, 
our setting is more general in that each subset may contain different number of
states and transitions between states are less restrictive. From the above
assumptions, we can rearrange states in $X$ such that 
\begin{align}
  Q_0 = \mbox{diag}\{Q_{0,1}, Q_{0,2}, \cdots, Q_{0,m}\}\,
  \label{q0-block}
\end{align}
  is a
block diagonal matrix consisting of $m$ submatrices $Q_{0,i}$, $1 \le i \le
m$, where $Q_{0,i}$ is an $n_i \times n_i$ matrix and defines a Markov chain
$\mathcal{C}_i$ on subset $X_i$.
We further assume that for each $1 \le i \le m$, Markov chain $\mathcal{C}_i$ is irreducible and therefore has a unique invariant measure $\pi_i$. 
Let $\pi^\epsilon$ be the unique invariant measure of Markov chain
$\mathcal{C}$ on $X$.
Given $1 \le i, j \le m$, we define
\begin{align}
  \bar{Q}(i,j) = \sum_{x \in X_i, y \in X_j} Q_1(x,y) \pi_i(x)\,. 
  \label{q-bar}
\end{align}
It is direct to verify that matrix $\bar{Q}$ in (\ref{q-bar}) defines an infinitesimal
generator (non-negative
off-diagonal elements with zero row sums) of Markov chain $\bar{\mathcal{C}}$
on $\bar{X} = \{1,2, \cdots, m\}$. We will call $\bar{\mathcal{C}}$ the
reduced Markov chain and assume it has a unique invariant measure $w$. 

The main aim of this paper is to consider several objects associated
with Markov chain $\mathcal{C}$ and their counterparts associated with Markov
chain $\bar{\mathcal{C}}$. For this purpose, we first introduce
the Kolmogorov backward equations 
\begin{align}
  \begin{split}
    \frac{d}{dt} \rho_t =  Q\rho_t = \big(\frac{1}{\epsilon} Q_0 + Q_1\big)
    \rho_t\,, \qquad  
    \frac{d}{dt} \bar{\rho}_t =  \bar{Q}\bar{\rho}_t \,, 
  \end{split}
    \label{fp-ms-intro} 
\end{align}
where $\rho_t : X \rightarrow \mathbb{R}$ and $\bar{\rho}_t : \bar{X}
\rightarrow \mathbb{R}$, $t \ge 0$.
These equations play an important role in understanding the dynamical
behaviors
of Markov chain $\mathcal{C}$ and $\bar{\mathcal{C}}$ \cite{norris-mc, pavliotis2008multiscale}.

We will also consider constants characterizing the speed of Markov
chain converging to equilibrium \cite{gross-lsi, Ledoux04spectralgap, ledoux-concetration-lsi}. 
Let $\mathbf{E}_{\pi^\epsilon}$,
  $\mbox{Var}_{\pi^\epsilon}$ denote the expectation and variance with respect
  to measure $\pi^\epsilon$ respectively. First recall the definition of Poincar\'e constant
and logarithmic Sobolev constant for Markov chain $\mathcal{C}$, which are defined as
\begin{align}
    &  \lambda_\epsilon =
  \inf_f\Big\{\frac{\mathcal{E}_\epsilon(f,f)}{\mbox{Var}_{\pi^\epsilon} f} 
  ~\Big|~\mbox{Var}_{\pi^\epsilon} f> 0\,,\, f : X \rightarrow \mathbb{R}\Big\} \label{poincare-const}
  \\
  &  \alpha_\epsilon =
  \inf_{f}\Big\{\frac{\mathcal{E}_\epsilon(f,f)}{\mbox{Ent}_{\pi^\epsilon} (f^2)}\,~\Big|~
  \mbox{Ent}_{\pi^\epsilon}(f^2) > 0\,,\, f : X \rightarrow \mathbb{R}\Big\}\,,
  \label{log-sobolev-const}
\end{align}
where the infima are taken among all non-constant functions, 
$\mathcal{E}_\epsilon$, $\mbox{Ent}_{\pi^\epsilon}$ are the Dirichlet form and
relative entropy with respect to $\pi^\epsilon$, defined as 
\begin{align}
  \mathcal{E}_\epsilon(f,g) =& -\langle f, Qg\rangle_{\pi^\epsilon} = - \Big\langle  f,
  (\frac{Q_0}{\epsilon} + Q_1)g\Big\rangle_{\pi^\epsilon}\,,\quad  f,\, g : X \rightarrow \mathbb{R}\,,\\
 \mbox{Ent}_{\pi^\epsilon}(f) =& \sum_{x \in X} f(x) \,\ln
  \frac{f(x)}{\mathbf{E}_{\pi^\epsilon}
  f}\, \pi^\epsilon(x)\,, \quad f : X \rightarrow \mathbb{R}^+\,.
\end{align}
For function $f : X \rightarrow \mathbb{R}$, we have
\begin{align}
  \mathcal{E}_\epsilon(f,f) =&  \frac{1}{2\epsilon} \sum_{i=1}^m \sum_{x,x' \in X_i}
(f(x') - f(x))^2 Q_{0,i}(x,x') \pi^\epsilon(x) \notag \\
& + \frac{1}{2} \sum_{i\neq j} \sum_{x\in X_i,\, y \in X_j} (f(y) -
f(x))^2 Q_1(x,y) \pi^\epsilon(x)\,,
\label{dirichlet-form-f-f-multiscale} 
\end{align}
which holds in both reversible and non-reversible case \cite{log-sob-diaconis}.

The modified logarithmic Sobolev constant is defined as
\begin{align}
  \gamma_\epsilon = \inf_{f} \Big\{ \frac{\mathcal{E}_\epsilon(f, \ln
  f)}{\mbox{Ent}_{\pi^\epsilon}(f)}~\Big|~ \mbox{Ent}_{\pi^\epsilon}(f) > 0\,, f : X \rightarrow \mathbb{R}^+ \Big\}\,,
    \label{mlsi-def}
\end{align}
where the infimum is taken among all non-constant and non-negative functions.
 It is known that these constants satisfy 
\begin{align}
4 \alpha_\epsilon \le \gamma_\epsilon \le 2 \lambda_\epsilon \,,
\label{order-diff-const-reversible}
\end{align}
in the reversible case, and 
\begin{align}
2 \alpha_\epsilon \le \gamma_\epsilon \le 2 \lambda_\epsilon \,, \quad
2\alpha_\epsilon \le \lambda_\epsilon\,,
\label{order-diff-const-nonreversible}
\end{align}
in the non-reversible case.
See \cite{mlsi-bobkov, caputo2009, log-sob-diaconis, Guionnet2002}
and references therein for more details.
Let $\bar{\mathcal{E}}$ denote the Dirichlet form of the reduced Markov chain $\bar{\mathcal{C}}$. 
Its Poincar\'e constant $\bar{\lambda}$, logarithmic Sobolev constant
$\bar{\alpha}$, as well as the modified logarithmic Sobolev constant $\bar{\gamma}$
can be defined similarly as in (\ref{poincare-const}), (\ref{log-sobolev-const}), (\ref{mlsi-def}), by replacing $\mathcal{E}_\epsilon$,
$\pi^\epsilon$, $Q$ with $\bar{\mathcal{E}}$, $w$, $\bar{Q}$ respectively.
Correspondingly, they satisfy the
inequality 
\begin{align}
4 \bar{\alpha} \le \bar{\gamma} \le 2 \bar{\lambda} \,,
\label{order-diff-const-bar-reversible}
\end{align}
in the reversible case, and 
\begin{align}
  2 \bar{\alpha} \le \bar{\gamma} \le 2 \bar{\lambda} \,, \quad 2\bar{\alpha}
  \le \bar{\lambda}\,,
\label{order-diff-const-bar-nonreversible}
\end{align}
in the non-reversible case.

Briefly speaking, in this paper we will establish the
convergence of $\rho_t$ to $\bar{\rho}_t$ in (\ref{fp-ms-intro}), and the
convergence of constants $\lambda_\epsilon$, $\alpha_\epsilon$,
$\gamma_\epsilon$ in (\ref{poincare-const}), (\ref{log-sobolev-const}), (\ref{mlsi-def}) to their
counterpart $\bar{\lambda}$, $\bar{\alpha}$ and $\bar{\gamma}$ respectively.
\subsection{Notations}
\label{subsec-notation}
In this subsection we collect some notations and definitions used in this paper. 
Let $\Omega$ be a finite set. For function $f
: \Omega \rightarrow \mathbb{R}$, 
\begin{align}
  |f|_{\infty}:=\max\limits_{x \in \Omega} |f(x)|\,, \quad |f|_2 :=
  \Big(\sum_{x \in \Omega} f^2(x)\Big)^{\frac{1}{2}}\,
\end{align}
are the $L^\infty$ norm and $L^2$ norm of $f$.  
Given a matrix $A$
of order $k \times l$, denote its 
infinity norm as $\|A\|_\infty$, i.e. $\|A\|_\infty = \sup\limits_{1\le i\le k}
\sum\limits_{j=1}^{l}|a_{ij}|$. 
For matrix $Q_1$ in (\ref{q-form}), we define 
\begin{align}
  Q_\infty := \|Q_1\|_{\infty} = \max_{x \in X} \sum_{x' \in X} |Q_1(x,x')| = 
2 \max_{x \in X}  \sum_{x' \neq x} Q_1(x,x')\,,
\label{Q-inf}
\end{align}
where we have used the fact that the off-diagonal entries of $Q_1$ are
non-negative, and $Q_1(x,x) = -\sum\limits_{x'\neq x} Q_1(x,x') < 0$, $\forall x \in X$.
From definition (\ref{q-bar}) of matrix $\bar{Q}$, it is direct to check $\|\bar{Q}\|_\infty \le Q_\infty$. 

Let $\mu$ be a probability measure over set $\Omega$, 
$L^2(\mu)$ is the Hilbert space consisting of all real functions on $\Omega$ with inner product 
\begin{align*}
  \langle f, g\rangle_{\mu} = \sum_{x \in \Omega} f(x)g(x)\mu(x)\,, \quad \forall f,g
  : \Omega \rightarrow \mathbb{R}\,,
\end{align*}
and its norm is denoted as $|\cdot|_{2, \mu}$.
We write 
\begin{align}
\mathbf{E}_\mu f = \sum_{x \in \Omega} f(x)\mu(x) \,, \quad 
  \quad \mbox{Var}_\mu f = \sum_{x \in \Omega} (f(x) - \mathbf{E}_\pi f)^2 \mu(x)\,,
\end{align}
as the expectation and the variance of function $f$ with respect to $\mu$. 

For Markov chain $\mathcal{C}_i$ whose infinitesimal generator is $Q_{0,i}$, we denote 
its Dirichlet form, Poincar\'e constant, logarithmic Sobolev constant,
modified logarithmic Sobolev constant as $\mathcal{E}_i$, $\lambda_i$, $\alpha_i$ and $\gamma_i$,
respectively. Also set 
$$\lambda_{\min} = \min\limits_i \lambda_i\,, \quad \alpha_{\min} =
\min\limits_i \alpha_i\,, \quad \gamma_{\min} = \min\limits_i \gamma_i\,.
$$ 

Given function $f: X \rightarrow \mathbb{R}$ and $1 \le i \le m$, $f(i,
\cdot)$ denotes the vector of length $n_i$ consisting of components $f(x)$ for $x \in
X_i$, while $\widetilde{f}$ denotes a function on $\bar{X}$, defined by $\widetilde{f}(i)
= \sum\limits_{x\in X_i} f(x) \pi_i(x)$, $1\le i \le m$.

We also need some notations when studying the general non-reversible case.
Define 
\begin{align}
  \begin{split}
  \Gamma := &\, \mbox{tr} (Q_1Q_1^T) = \sum_{x \in X}\sum_{y \in X}
  Q_1(y,x)^2\,, \\
  d := & \max_{x \in X} \Big|\Big\{y \in X ~|~ Q_1(x,y) \neq
0\Big\}\Big|\,,
\end{split}
\end{align}
where $|\cdot|$ denotes the cardinality of a given set.
$\sigma_i$ and $\bar{\sigma}$ denote the smallest nonzero singular value of
matrix $Q_{0,i}$ and $\bar{Q}$, respectively. Also set $\sigma_{\min} = \min\limits_i \sigma_i$.

The paper is organized as follows. Section~\ref{sec-reversible} is devoted to
obtain several asymptotic results when Markov chain $\mathcal{C}$ is reversible.
The general Markov chain without reversibility assumption is studied in Section~\ref{sec-general}. 
In Section~\ref{sec-conclusion}, we discuss our results and make conclusions. 
Appendix~\ref{app-sec-1} collects some useful facts related to 
continuous-time Markov chain. Appendix~\ref{app-sec-2} contains formal
arguments which motivates our asymptotic results.
\section{Asymptotic analysis : reversible case}
\label{sec-reversible}
In this section, we establish several asymptotic convergence results under the assumption
that Markov chain $\mathcal{C}$ is reversible. 
\subsection{Invariant measure}
We start with the invariant measure $\pi^\epsilon$. 
Taking the structure of matrix $Q$ in (\ref{q-form}), (\ref{q0-block})
into consideration, the detailed balance condition reads
\begin{align}
  \begin{split}
  & \pi^\epsilon(x) Q_{0,i}(x,x') = \pi^\epsilon(x') Q_{0,i}(x',x)\,, \quad
  \mbox{if}\quad x, x' \in X_i\,, \\
  & \pi^\epsilon(x) Q_1(x,y) = \pi^\epsilon(y) Q_1(y,x)\,, \quad
  \mbox{if}\quad x \in X_i\,,\, y \in
  X_j\,, i \neq j\,. 
\end{split}
\label{detail-balance-c}
\end{align}
Since we assume Markov chain
$\mathcal{C}_i$ has a unique invariant measure, the first equation above implies
that $\mathcal{C}_i$ is also reversible, for $1 \le i \le m$,
and $\exists w^\epsilon(i) > 0$ s.t. 
\begin{align}
  \pi^\epsilon(x) = w^\epsilon(i) \pi_i(x)\,, \quad \mbox{for}~ x \in X_i\,.
  \label{pi-w}
\end{align}
 We have
\begin{align}
  \sum_{i=1}^m w^\epsilon(i) = \sum_{x \in X} \pi^\epsilon(x) = 1\,.
  \label{w-sum-1}
\end{align}
 Substituting
relation (\ref{pi-w}) 
into the second equation of (\ref{detail-balance-c}) and summing up all states
$x \in X_i, y \in X_j$, we obtain
\begin{align}
  w^\epsilon(i) \bar{Q}(i,j) = w^\epsilon(j) \bar{Q}(j,i)\,, \quad 1 \le i \neq j \le m\,,
  \label{detail-balance-c-bar}
\end{align}
where matrix $\bar{Q}$ is defined in (\ref{q-bar}). Equation
(\ref{w-sum-1}) and (\ref{detail-balance-c-bar}) imply that $w^\epsilon$ 
coincides with the invariant measure $w$ of Markov chain $\bar{\mathcal{C}}$
and furthermore, $\bar{\mathcal{C}}$ is reversible with respect to $w$.
From (\ref{pi-w}) we also know that $\pi^\epsilon$ is independent of parameter $\epsilon$.
In the following of this section we will denote it as $\pi$ for simplicity.
\subsection{Kolmogorov backward equation}
We consider the Kolmogorov backward equation 
\begin{align}
  \frac{d}{dt} \rho_t = & Q\rho_t = \big(\frac{1}{\epsilon} Q_0 + Q_1\big)
  \rho_t \label{fp-ms} 
\end{align}
with initial condition $\rho_0$ ($\rho_0$ can be negative), or more explicitly, 
\begin{align}
  \frac{d}{dt} \rho_t(x) = & \frac{1}{\epsilon} \sum_{x'\neq x,\, x' \in X_i }
  \big(\rho_t(x') - \rho_t(x)\big) Q_{0,i}(x, x') + \sum_{y \not\in X_i} \big(\rho_t(y) -
  \rho_t(x)\big) Q_1(x, y) \,,
\label{fp-ms-x} 
\end{align}
for $x \in X_i$, $1\le i \le m$. 
Multiplying both sides of (\ref{fp-ms-x}) by $\pi_i(x)$, summing up states $x \in X_i$,
and noticing that $Q_{0,i}^T \pi_i = 0$, we obtain the equation of
$\widetilde{\rho}_t(i) = \sum\limits_{x \in X_i} \rho_t(x) \pi_i(x)$ as 
\begin{align}
  \frac{d}{dt} \widetilde{\rho}_t(i) = \sum_{x \in X_i}\sum_{y \not\in X_i}
  (\rho_t(y) - \rho_t(x)) Q_1(x, y) \pi_i(x)\,, \quad 1 \le i \le m\,.
  \label{marginal-fp}
\end{align}

We also introduce the Kolmogorov backward equation of the reduced
Markov chain $\bar{\mathcal{C}}$
\begin{align}
  \frac{d}{dt} \bar{\rho}_t = & \bar{Q} \bar{\rho}_t 
  = \sum_{j \neq i} (\bar{\rho}_t(j) - \bar{\rho}_t(i)) \bar{Q}(i,
  j)\,
\label{averaged-rho-eqn}
\end{align}
with initial condition $\bar{\rho}_0 = \widetilde{\rho}_0$, where matrix $\bar{Q}$
is defined in (\ref{q-bar}). We have 
\begin{theorem}
  Assume Markov chain $\mathcal{C}$ is reversible. 
  Consider functions $\rho_t$, $\widetilde{\rho}_t$ and $\bar{\rho}_t$, which
  are solutions of equation (\ref{fp-ms}), (\ref{marginal-fp}) and
  (\ref{averaged-rho-eqn}), respectively. For $t \ge 0$, we have
\begin{align}
 |\rho_t(\cdot) -
\widetilde{\rho}_t(i)\mathbf{1}\big|_{2,\pi_i} 
\le \Big(e^{-\frac{\lambda_i t}{\epsilon}} + \frac{2\epsilon}{\lambda_i}
 Q_\infty\Big) |\rho_0|_{\infty}\,,
 \label{thm-1-eqn-1}
\end{align}
\begin{align}
  |\widetilde{\rho}_t - \bar{\rho}_t|_{2,w} 
  \le & \frac{Q_\infty|\rho_0|_{\infty}}{\lambda_{\min}}
 \Big(\min\Big\{\frac{1}{\min\limits_{i,x \in X_i} \pi_i(x)},
\frac{m}{2}\Big\}\Big)^{\frac{1}{2}}
\Big(\frac{2Q_\infty}{\bar{\lambda}} + 1\Big) \epsilon\,,
 \label{thm-1-eqn-2}
\end{align}
 where constants involved are defined in Section~\ref{sec-intro}.
\label{thm-0}
\end{theorem}

Before entering the proof, we would like to reinterpret the results of Theorem~\ref{thm-0} by
considering the corresponding Markov chain processes.
Let $x_t \in X$ and $\bar{x}_t \in \bar{X}$ be the Markov chain $\mathcal{C}$ and
$\bar{\mathcal{C}}$, respectively.
Given function $f : X \rightarrow \mathbb{R}$ and defining $\widetilde{f}(i) =
\sum\limits_{x \in X_i} f(x)\pi_i(x)$ as before, we consider quantities 
\begin{align}
   f_t(x) = \mathbf{E}(f(x_t)~|~x_0 = x)\,,\quad 
  \widetilde{f}_t(i) = \mathbf{E}(f(x_t)~|~x_0 \sim \pi_i)\,, \quad x \in X\,,
  \label{exp-1}
\end{align}
then we know $f_t$
satisfies equation (\ref{fp-ms}) with initial condition $f_0 = f$, while
$\widetilde{f}_t(i)$ satisfies (\ref{marginal-fp}) with $\rho_t$ replaced by
$f_t$.
  Similarly define 
\begin{align}
  & \bar{f}_t(i) = \mathbf{E}(\widetilde{f}(\bar{x}_t)~|~\bar{x}_0 = i)\,,
  \quad 1 \le i \le m\,,
  \label{exp-2}
\end{align}
then  
$\bar{f}$ satisfies (\ref{averaged-rho-eqn}) with initial condition
$\bar{f}_0 = \widetilde{f}_0 = \widetilde{f}$. Theorem~\ref{thm-0} 
implies
\begin{corollary}
  Consider reversible Markov chains $x_t \in X$ and $\bar{x}_t \in \bar{X}$ defined by infinitesimal generator $Q$ and $\bar{Q}$, respectively.
  Given $f : X \rightarrow \mathbb{R}$, define the quantities $f_t, \widetilde{f}_t, \bar{f}_t$ by (\ref{exp-1}) and
  (\ref{exp-2}). We have $\forall\, t \ge 0$, 
  \begin{align}
    \begin{split}
    & |f_t(i, \cdot) -
\widetilde{f}_t(i)\mathbf{1}\big|_{2,\pi_i} 
\le \Big(e^{-\frac{\lambda_i t}{\epsilon}} + \frac{2\epsilon}{\lambda_i}
 Q_\infty\Big) |f|_{\infty}\,, \quad 1 \le i \le m\,,\\
 & |\widetilde{f}_t - \bar{f}_t|_{2,w} 
  \le
   \frac{Q_\infty|f|_{\infty}}{\lambda_{\min}}
 \Big(\min\Big\{\frac{1}{\min\limits_{i,x \in X_i} \pi_i(x)},
\frac{m}{2}\Big\}\Big)^{\frac{1}{2}}
\Big(\frac{2Q_\infty}{\bar{\lambda}} + 1\Big) \epsilon\,.
\end{split}
\end{align}
\label{corol-1}
\end{corollary}

Now consider a probability measure $\mu$ on $X$ and define probability measure
$\widetilde{\mu}$ on $\bar{X}$ by $\widetilde{\mu}(i) = \sum\limits_{x \in X_i}
\mu(x)$, $1 \le i \le m$. Also define
\begin{align}
  \begin{split}
   \mu_t(x) &= \mathbf{P}(x_t = x\,|\,x_0 \sim \mu)\,,\quad x \in X\,,\\
    \widetilde{\mu}_t(i) &= \mathbf{P}(x_t \in X_i\,|\,x_0 \sim \mu) = \sum_{x
  \in X_i} \mu_t(x) \,,\quad 1 \le i \le m\,, \\
    \bar{\mu}_t(i) &= \mathbf{P}(\bar{x}_t = i\,|\,\bar{x}_0 \sim
  \widetilde{\mu}) \,, \quad 1 \le i \le m\,,
  \end{split}
  \label{prob-measures}
\end{align}
and the probability densities with respect to invariant measures $\pi$ and $w$ 
\begin{align}
  \rho_t = \frac{d\mu_t}{d\pi}\,, \quad \widetilde{\rho}_t =
  \frac{d\widetilde{\mu}_t}{d w}\,,\quad \bar{\rho}_t =
  \frac{d\bar{\mu}}{d w}\,.
  \label{prop-measures-density}
\end{align}
Recalling the detailed balance condition (\ref{detail-balance-c}), we can
check that functions $\rho_t$,
$\widetilde{\rho}_t$ and $\bar{\rho}_t$ satisfy equation (\ref{fp-ms}),
(\ref{marginal-fp}) and (\ref{averaged-rho-eqn}), respectively (see
Appendix~\ref{app-sec-1}).
Therefore Theorem~\ref{thm-0} implies
\begin{corollary}
  Consider reversible Markov chains $x_t \in X$ and $\bar{x}_t \in \bar{X}$ defined by
  infinitesimal generator $Q$ and $\bar{Q}$, respectively. Given probability
  measure $\mu$ on space $X$. Let $\rho_t$, $\widetilde{\rho}_t$
  and $\bar{\rho}_t$ be the density of probability measures
  defined by (\ref{prob-measures}) and  (\ref{prop-measures-density}). For $t\ge 0$, we have 
  \begin{align}
    \begin{split}
    & |\rho_t(i, \cdot) -
\widetilde{\rho}_t(i)\mathbf{1}\big|_{2,\pi_i} 
\le \Big(e^{-\frac{\lambda_i t}{\epsilon}} + \frac{2\epsilon}{\lambda_i}
 Q_\infty\Big) |\rho_0|_{\infty}\,, \quad 1 \le i \le m\,,\\
 & |\widetilde{\rho}_t - \bar{\rho}_t|_{2,w} 
 \le \frac{Q_\infty|\rho_0|_{\infty}}{\lambda_{\min}}
 \Big(\min\Big\{\frac{1}{\min\limits_{i,x \in X_i} \pi_i(x)},
\frac{m}{2}\Big\}\Big)^{\frac{1}{2}}
\Big(\frac{2Q_\infty}{\bar{\lambda}} + 1\Big) \epsilon\,.
\end{split}
\end{align}
\label{corol-2}
\end{corollary}
{\em Proof of Theorem~\ref{thm-0}} :
\begin{enumerate}
  \item
  We start with the first inequality (\ref{thm-1-eqn-1})
  concerning $\rho_t$ and $\widetilde{\rho}_t$.
For $\rho_t$ satisfying (\ref{fp-ms}), we know 
$|\rho_t|_{\infty} \le |\rho_0|_{\infty}$, $t \ge 0$ (can be easily seen
from (\ref{exp-1})).
For the right hand side of (\ref{marginal-fp}), we have 
\begin{align*}
  &  \Big|\sum_{x \in X_i} \sum_{y\not\in X_i} (\rho_t(y) - \rho_t(x)) Q_1(x, y)
\pi_i(x)\Big| \\
=&  \Big|\sum_{x \in X_i} \sum_{y\in X} \rho_t(y) Q_1(x, y) \pi_i(x)\Big|
\le  |\rho_t|_\infty \max_{x \in X} \sum_{y \in X} |Q_1(x,y)| \le
Q_\infty|\rho_0|_{\infty}\,.
\end{align*}
Therefore (\ref{marginal-fp}) implies 
\begin{align}
|\widetilde{\rho}_t(i) - \widetilde{\rho}_s(i)| \le
|t-s|Q_{\infty} |\rho_0|_{\infty}\,, \quad 1 \le i \le m\,.
\label{bound-1}
\end{align}

For equation (\ref{fp-ms}) which is written in matrix form, using variation of constants
formula, we can obtain
\begin{align}
  \rho_t = e^{(t - s)Q_0 / \epsilon} \rho_s + \int_s^t e^{(t - r)Q_0 /
  \epsilon} Q_1 \rho_r \,dr\,, \quad 0 \le s \le t.
  \label{rho-variaion-constant}
\end{align}
Since $e^{(t - r)Q_0 / \epsilon}$ is a stochastic matrix, we have 
\begin{align}
  \begin{split}
  & \Big|\int_s^t e^{(t - r)Q_0 / \epsilon} Q_1 \rho_r\, dr\Big|_{\infty} \\
  \le &
  \int_s^t \big|Q_1 \rho_r\big|_{\infty} dr\le 
  \int_s^t \big\|Q_1\big\|_{\infty} \big|\rho_r|_{\infty} dr
  \le (t - s) Q_\infty |\rho_0|_{\infty}\,.
\end{split}
  \label{bound-2}
\end{align}
For the first term on the right hand side of (\ref{rho-variaion-constant}), noticing
$Q_0$ (therefore also $e^{(t - s)Q_0 / \epsilon}$) is a block diagonal matrix
and applying Poincar\'e inequality \cite{mlsi-bobkov, caputo2009, log-sob-diaconis}, we deduce
\begin{align}
  \begin{split}
  & \big|e^{(t - s)Q_{0,i} / \epsilon} \rho_s(i, \cdot) -
 \widetilde{\rho}_s(i)\mathbf{1}\big|_{2,\pi_i}  \\
 = & \big|e^{(t - s)Q_{0,i} / \epsilon} (\rho_s(i,
  \cdot) - \widetilde{\rho}_s(i)\mathbf{1})\big|_{2,\pi_i} 
  \le 
  e^{-\frac{(t-s)\lambda_i}{\epsilon}} 
\big|\rho_s(i, \cdot) - \widetilde{\rho}_s(i)\mathbf{1}\big|_{2,\pi_i}\,,
\end{split}
\label{bound-3}
\end{align}
where $1 \le i\le m$,
$\mathbf{1}$ denotes the constant vector over subset $X_i$ and
$\rho_s(i, \cdot)$ denotes the vector consisting of $\rho_s(x)$ for $x \in
X_i$. Combining estimates (\ref{bound-1})-(\ref{bound-3}) together, we have 
\begin{align*}
  & |\rho_t(i, \cdot) - \widetilde{\rho}_t(i)\mathbf{1}\big|_{2,\pi_i} \\
 \le&  |\widetilde{\rho}_t(i) - \widetilde{\rho}_s(i)| + 
  e^{-\frac{(t-s)\lambda_i}{\epsilon}} 
\big|\rho_s(i, \cdot) - \widetilde{\rho}_s(i)\mathbf{1}\big|_{2,\pi_i} 
 + (t - s) Q_\infty |\rho_0|_{\infty} \\
\le & e^{-\frac{(t-s)\lambda_i}{\epsilon}} 
\big|\rho_s(i, \cdot) - \widetilde{\rho}_s(i)\mathbf{1}\big|_{2,\pi_i} 
+ 2(t - s) Q_\infty |\rho_0|_{\infty} \,.
\end{align*}
Fix index $i$ and define $G(t) = |\rho_t(i, \cdot) -
\widetilde{\rho}_t(i)\mathbf{1}\big|_{2,\pi_i}$. We subtract $G(s)$ 
and then divide $(t-s)$ on both sides of the inequality above. Let $t \rightarrow s+$, we obtain
\begin{align}
  \frac{d^+ G(t)}{dt} + \frac{\lambda_i}{\epsilon} G(t) \le 
  2 Q_\infty |\rho_0|_{\infty} \,, \quad t \ge 0\,.
\end{align}
Gronwall's inequality then implies 
\begin{align}
 |\rho_t(i, \cdot) -
\widetilde{\rho}_t(i)\mathbf{1}\big|_{2,\pi_i} 
\le & e^{-\frac{\lambda_i t}{\epsilon}} |\rho_0(i, \cdot) - \widetilde{\rho}_0(i)\mathbf{1}\big|_{2,\pi_i}
 + \frac{2\epsilon}{\lambda_i} Q_\infty |\rho_0|_{\infty} \notag \\
 \le &\big(e^{-\frac{\lambda_i t}{\epsilon}} + \frac{2\epsilon}{\lambda_i}
 Q_\infty\big) |\rho_0|_{\infty}\,, \label{1st-part-thm-0}
\end{align}
where we have used 
\begin{align*}
  &  |\rho_0(i, \cdot) -
\widetilde{\rho}_0(i)\mathbf{1}\big|_{2,\pi_i}  \\
=& \Big[\sum_{x \in X_i} \Big(\rho_0(x) - \sum_{x'\in X_i}
  \rho_0(x')\pi_i(x')\Big)^2 \pi_i(x) \Big]^{\frac{1}{2}}\\
  =& \Big[\sum_{x \in X_i} \rho_0(x)^2\pi_i(x) - \Big(\sum_{x\in X_i}
  \rho_0(x)\pi_i(x)\Big)^2 \Big]^{\frac{1}{2}}
\le |\rho_0|_{\infty}\,.
\end{align*}
\item
Now we turn to the second inequality (\ref{thm-1-eqn-2}). 
First notice that the equation of $\widetilde{\rho}_t$ in
(\ref{marginal-fp}) can be rewritten as 
\begin{align}
  \frac{d}{dt} \widetilde{\rho}_t(i) 
  = & \sum_{j\neq i} \sum_{x \in X_i,\,y\in X_j} (\rho_t(y) - \rho_t(x))
  Q_1(x,y) \pi_i(x) \notag \\
   = & \sum_{j \neq i} (\widetilde{\rho}_t(j) - \widetilde{\rho}_t(i))
  \bar{Q}(i, j)\notag \\
 & +   
  \sum_{j \neq i} \sum_{x\in X_i,\,y \in X_j} \Big[\big(\rho_t(y) - \widetilde{\rho}_t(j)\big) -
    \big(\rho_t(x)
  - \widetilde{\rho}_t(i)\big)\Big] Q_1(x, y) \pi_i(x)\notag \\
  = & \sum_{j \neq i}  (\widetilde{\rho}_t(j) -
  \widetilde{\rho}_t(i)) \bar{Q}(i, j) + \phi_t(i) = \bar{Q}
  \widetilde{\rho}_t + \phi_t\,,
  \label{marginal-fp-1}
\end{align}
where 
\begin{align}
  \phi_t(i) = &
  \sum_{j \neq i} \sum_{x\in X_i,\,y \in X_j} \Big[\big(\rho_t(y) - \widetilde{\rho}_t(j)\big) -
    \big(\rho_t(x)
  - \widetilde{\rho}_t(i)\big)\Big] Q_1(x, y) \pi_i(x)\notag \\
  =& 
  \sum_{j=1}^m \sum_{x\in X_i,\,y \in X_j} \big(\rho_t(y) -
  \widetilde{\rho}_t(j)\big) Q_1(x, y) \pi_i(x)\notag\,,
  \end{align}
since the row sums of $Q_1$ are zero. Using detailed
balance condition (\ref{detail-balance-c}), we can obtain
\begin{align*}
  \mathbf{E}_w\phi_t =&
\sum\limits_{i=1}^m \phi_t(i) w(i)  \\
=&  \sum_{i=1}^m \Big[\sum_{j=1}^m\sum_{x \in X_i, y \in X_j}
    \Big(\rho_t(y) - \widetilde{\rho}_t(j)\Big)Q_1(x,y)\pi_i(x)\Big] w(i) \\ 
    =& \sum_{x\in X} \sum_{y \in X} \Big(\rho_t(y) -
  \widetilde{\rho}_t(j)\Big)\pi(y) Q_1(y,x) = 0\,. 
\end{align*}

We also need to estimate $|\phi_t|_{2,w}$. On one hand, applying inequality
(\ref{1st-part-thm-0}) , we can deduce a pointwise estimate
\begin{align}
  \big|\rho_t(i, x) - \widetilde{\rho}_t(i)\big| \le \sqrt{\frac{1}{\min\limits_{x' \in X_i}
  \pi_i(x')}}
\Big(e^{-\frac{\lambda_i t}{\epsilon}} + \frac{2\epsilon}{\lambda_i}
 Q_\infty\Big) |\rho_0|_{\infty}\,, \quad \forall x \in X_i\,.
 \label{rho-point-error}
\end{align} 
Therefore 
\begin{align}
  |\phi_t|_{2,w} \le & \max_i |\phi_t(i)| \le Q_\infty \max_{i} |\rho_t(i,
  \cdot) - \widetilde{\rho}_t(i)|_{\infty} \notag \\
  \le & Q_\infty \sqrt{\frac{1}{\min\limits_{i, x \in X_i} \pi_i(x)}}
\Big( e^{-\frac{\lambda_{\min} t}{\epsilon}} +
\frac{2\epsilon}{\lambda_{\min}} Q_\infty\Big) |\rho_0|_{\infty}\,.
 \label{phi-upper-bound-1}
\end{align}
On the other hand, we can avoid using pointwise estimate (\ref{rho-point-error}) and
compute
\begin{align*}
|\phi_t|_{2,w}^2 
= & \sum_{i=1}^m \Big[\sum_{j=1}^m\sum_{x \in X_i, y \in X_j}
    \Big(\rho_t(y) - \widetilde{\rho}_t(j)\Big)Q_1(x,y)\pi_i(x)\Big]^2 w(i) \\ 
= & \sum_{i=1}^m \Big[\sum_{j=1}^m\sum_{x \in X_i, y \in X_j}
  \Big(\rho_t(y) - \widetilde{\rho}_t(j)\Big)Q_1(y,x)\pi_j(y)w(j)\Big]^2
  \frac{1}{w(i)} \\ 
\le&  \Big[\sum_{j=1}^m\sum_{y \in X_j}
  \Big(\rho_t(y) - \widetilde{\rho}_t(j)\Big)^2
  \pi_j(y)w(j)\Big] \sum_{i=1}^m\Big[\sum_{j=1}^m\sum_{y \in X_j}
  \Big(\sum_{x \in X_i} Q_1(y,x)\Big)^2 \pi_j(y)w(j)\Big] \frac{1}{w(i)} \\ 
\le&  \Big[\sum_{j=1}^m
   |\rho_t(j, \cdot) -
\widetilde{\rho}_t(j)\mathbf{1}\big|_{2,\pi_j}^2 w(j)\Big]
  \sum_{i=1}^m\sum_{j=1}^m\sum_{y \in X_j}
  \Big|\sum_{x \in X_i} Q_1(y,x)\Big|\cdot \Big|\sum_{x \in X_i} Q_1(x,y) \pi_i(x)\Big| \\ 
  \le& \frac{Q_{\infty}}{2} \Big[\sum_{j=1}^m
   |\rho_t(j, \cdot) -
\widetilde{\rho}_t(j)\mathbf{1}\big|_{2,\pi_j}^2 w(j)\Big]
  \sum_{i=1}^m \sum_{x \in X_i} \sum_{y \in X} |Q_1(x,y)| \pi_i(x) \\ 
  \le & \frac{m Q_{\infty}^2}{2} \Big[\sum_{j=1}^m
   |\rho_t(j, \cdot) -
\widetilde{\rho}_t(j)\mathbf{1}\big|_{2,\pi_j}^2 w(j)\Big]\,, 
\end{align*}
where we have used detailed balance condition (\ref{detail-balance-c}) and 
relation (\ref{Q-inf}). Together with (\ref{phi-upper-bound-1}), we could
deduce
\begin{align}
|\phi_t|_{2,w}
\le &
 \Big(\min\Big\{\frac{1}{\min\limits_{i,x \in X_i} \pi_i(x)},
\frac{m}{2}\Big\}\Big)^{\frac{1}{2}}
Q_\infty \Big(e^{-\frac{\lambda_{\min} t}{\epsilon}}
+ \frac{2\epsilon}{\lambda_{\min}}
 Q_\infty\Big) |\rho_0|_{\infty}\,.
 \label{phi-upper-bound-2}
\end{align}

Now subtract (\ref{marginal-fp-1}) by equation (\ref{averaged-rho-eqn}), we
obtain
\begin{align}
  \frac{d}{dt} (\widetilde{\rho}_t - \bar{\rho}_t) = \bar{Q}
  (\widetilde{\rho}_t - \bar{\rho}_t)+ \phi_t\,,
\end{align}
together with initial condition $\bar{\rho}_0 = \widetilde{\rho}_0$.
Therefore we have 
\begin{align}
  \widetilde{\rho}_t - \bar{\rho}_t = \int_0^t e^{(t-s)\bar{Q}} \phi_s ds\,.
\end{align}
Since $\mathbf{E}_w\phi_s = 0$, Poincar\'e inequality implies
\begin{align}
|e^{(t-s)\bar{Q}} \phi_s|_{2,w} \le e^{-\bar{\lambda}(t-s)}
|\phi_s|_{2,w}\,.
\end{align}
Therefore, using (\ref{phi-upper-bound-2}), we
have 
\begin{align*}
  |\widetilde{\rho}_t - \bar{\rho}_t|_{2,w} 
  \le & \int_0^t |e^{(t-s)\bar{Q}} \phi_s|_{2,w} ds
  \le \int_0^t e^{-\bar{\lambda}(t-s)}
|\phi_s|_{2,w} ds \\
\le & \Big(\min\Big\{\frac{1}{\min\limits_{i,x \in X_i} \pi_i(x)},
\frac{m}{2}\Big\}\Big)^{\frac{1}{2}}
Q_\infty
\int_0^t e^{-\bar{\lambda}(t-s)}
|\rho_0|_{\infty} \Big(e^{-\frac{\lambda_{\min} s}{\epsilon}} +
\frac{2\epsilon}{\lambda_{\min}} Q_\infty\Big) ds \\
\le & \frac{Q_\infty|\rho_0|_{\infty}}{\lambda_{\min}}
 \Big(\min\Big\{\frac{1}{\min\limits_{i,x \in X_i} \pi_i(x)},
\frac{m}{2}\Big\}\Big)^{\frac{1}{2}}
\Big(\frac{2Q_\infty}{\bar{\lambda}} + 1\Big) \epsilon\,.  \qquad \endproof
\end{align*}
\end{enumerate}

\subsection{Poincar\'e constant, (modified) logarithmic Sobolev constants}
In this subsection we consider the asymptotic behavior of the Poincar\'e constant
$\lambda_\epsilon$, logarithmic Sobolev
constant $\alpha_\epsilon$, and modified logarithmic Sobolev constant $\gamma_\epsilon$
defined in (\ref{poincare-const}), (\ref{log-sobolev-const})
and (\ref{mlsi-def}), respectively. We will use the fact that the
infima in the definitions can be achieved by some extreme functions. 
Also notice that in the reversible case, as a generalization of (\ref{dirichlet-form-f-f-multiscale}), we have 
\begin{align}
  \mathcal{E}_\epsilon(f,g)
  =& \frac{1}{2\epsilon} \sum_{i=1}^m \sum_{x,x' \in X_i}
(f(x') - f(x)) (g(x') - g(x)) Q_{0,i}(x,x') \pi(x) \notag \\
& + \frac{1}{2} \sum_{i\neq j} \sum_{x\in X_i,\, y \in X_j} (f(y) -
f(x))(g(y) - g(x)) Q_1(x,y) \pi(x)\,,
\label{dirichlet-form-f-g-multiscale} 
\end{align}
for all $f, g : X \rightarrow \mathbb{R}$.
See \cite{mlsi-bobkov, log-sob-diaconis} and Appendix~\ref{app-sec-1} for more details. 

We start with the Poincar\'e constant. 
\begin{theorem}
  Assume Markov chain $\mathcal{C}$ is reversible and $\epsilon\le 1$. Let
  $\lambda_\epsilon, \bar{\lambda}$ be the Poincar\'e
  constants of Markov chain $\mathcal{C}$ and $\bar{\mathcal{C}}$ corresponding to infinitesimal generator $Q$ and
  $\bar{Q}$, respectively. We have 
\begin{align*}
  \frac{\bar{\lambda}}{(1 + \epsilon^{\frac{1}{2}})^2} \Big(1 - 
  \frac{2\bar{\lambda}\,\epsilon^{\frac{1}{2}}}{\lambda_{\min}}\Big)
- \frac{\bar{\lambda}Q_\infty}{\lambda_{\min}} \epsilon^{\frac{1}{2}} 
\le \lambda_\epsilon \le \bar{\lambda}\,.
\end{align*}
\label{thm-1}
\end{theorem}
\begin{proof}
Recall the Poincar\'e constant defined in (\ref{poincare-const})
\begin{align}
    &  \lambda_\epsilon =
  \inf_f\Big\{\frac{\mathcal{E}_\epsilon(f,f)}{\mbox{Var}_\pi f} 
  ~\Big|~\mbox{Var}_\pi f> 0\,,\, f : X \rightarrow \mathbb{R}\Big\}\,,
  \label{poincare-const-repeat}
\end{align}
and the Dirichlet form $\mathcal{E}_\epsilon$ in (\ref{dirichlet-form-f-g-multiscale}).
\begin{enumerate}
  \item
  First choose function $f : X \rightarrow \mathbb{R}$, s.t. $f(x) = g(i)$
  for $x \in X_i$, where $g$ is a function on $\bar{X}$.
  Using the fact $\pi(x) = \pi_i(x) w(i)$ when $x \in X_i$, 
from (\ref{dirichlet-form-f-g-multiscale}) we know
\begin{align*}
  \mathcal{E}_\epsilon(f,f) = \frac{1}{2} \sum_{1\le i,j \le m} (g(j) - g(i))^2
  \bar{Q}(i,j) w(i) = \bar{\mathcal{E}} (g,g)\,.
\end{align*}
It is also straightforward to check $\mathbf{E}_\pi f = \mathbf{E}_w g$ and
$\mbox{Var}_\pi f = \mbox{Var}_w g$. Allowing $g$ to vary among all functions from
$\bar{X}$ to $\mathbb{R}$, we obtain
$\lambda_\epsilon \le \bar{\lambda}$, 
i.e. the upper bound of the theorem. 
\item 
For the lower bound, we assume the minimum in (\ref{poincare-const-repeat}) is obtained by
function $f$, i.e.
\begin{align*}
  \frac{\mathcal{E}_\epsilon(f,f)}{\mbox{Var}_\pi f} = \lambda_\epsilon  \le \bar{\lambda}\,.
\end{align*}
The estimation of $\lambda_\epsilon$ can be obtained if we could estimate
$\mathcal{E}_\epsilon(f,f)$ and the variance of $f$.

From (\ref{dirichlet-form-f-g-multiscale}), we easily obtain 
\begin{align}
  \frac{1}{2\epsilon} \sum_{i=1}^m\Big[\sum_{x, x' \in X_i} \big(f(x') -
f(x)\big)^2 Q_{0,i}(x,x') \pi_i(x)\Big] w(i) \le \bar{\lambda} \mbox{Var}_\pi f\,.
  \label{eps-term-1}
\end{align}
Applying Poincar\'e inequality to Markov chain $\mathcal{C}_i$ for each fixed
$i$, we obtain
\begin{align}
  \sum_{i=1}^m \Big[\sum_{x\in X_i} \Big(f(x) - \widetilde{f}(i) \Big)^2\pi_i(x)\Big]w(i) 
  \le \frac{\bar{\lambda}}{\lambda_{\min}} \mbox{Var}_\pi f \, \epsilon\,,
  \label{f-tildef-ineq}
\end{align}
where $\widetilde{f}(i) = \sum\limits_{x\in X_i}f(x)\pi_i(x)$.
 Using the elementary inequality
\begin{align}
  &(1 +\epsilon^{\frac{1}{2}}) c^2 \ge (a + b + c)^2 - 2(1 +
  \epsilon^{-\frac{1}{2}}) (a^2 + b^2) \,,\quad \forall a,b,c \in \mathbb{R} \,,
  \label{basic-ineq-1}
\end{align}
and the detailed balance condition (\ref{detail-balance-c}), we can estimate  
the Dirichlet form $\mathcal{E}_\epsilon$ in (\ref{dirichlet-form-f-g-multiscale})
\begin{align}
  & \mathcal{E}_\epsilon(f,f) \notag \\
  \ge & ~\frac{1}{2} \sum_{i \neq j} \sum_{x \in X_i,\,y \in
X_j} \big(f(y) - f(x)\big)^2 Q_1(x,y) \pi(x) \notag  \\
  \ge &
\frac{1}{1 + \epsilon^{\frac{1}{2}}} \frac{1}{2} \sum_{i\neq j} \big(\widetilde{f}(j) -
  \widetilde{f}(i)\big)^2 \bar{Q}(i,j) w(i) \notag \\
  &- 
   \epsilon^{-\frac{1}{2}} \sum_{i \neq j} \sum_{x\in X_i,\,y\in X_j}
    \Big[\big(f(x) - \widetilde{f}(i)\big)^2 +
    \big(\widetilde{f}(j)-f(y)\big)^2\Big] Q_1(x,y) \pi(x) \notag \\
  = &
    \frac{1}{1 + \epsilon^{\frac{1}{2}}} \frac{1}{2} \sum_{i\neq j} \big(\widetilde{f}(j) -
  \widetilde{f}(i)\big)^2 \bar{Q}(i,j) w(i)
  - 2 \epsilon^{-\frac{1}{2}} \sum_{i \neq j} \sum_{x\in X_i,\,y\in X_j}
    \big(f(x) - \widetilde{f}(i)\big)^2 Q_1(x,y) \pi(x)\,. \label{dirichlet-ineq} 
\end{align}
Recalling the definition of $Q_\infty$ in (\ref{Q-inf})
and applying (\ref{f-tildef-ineq}), we can estimate the second term on the
right hand side of (\ref{dirichlet-ineq}) 
 and obtain 
\begin{align}
  & \sum_{i \neq j} \sum_{x\in X_i,\,y \in X_j}
    \big(f(x) - \widetilde{f}(i)\big)^2 Q_1(x,y) \pi(x)\notag \\
    \le & \frac{Q_\infty}{2}
    \sum_{i=1}^m \sum_{x \in X_i}
    (f(x) - \widetilde{f}(i))^2  \pi(x)
    \le \frac{\bar{\lambda}Q_\infty}{2\lambda_{\min}} \mbox{Var}_\pi f \,
    \epsilon\,.
    \label{d-form-1st-term-ineq}
\end{align}
To estimate the variance of $f$, we apply inequality (\ref{f-tildef-ineq}),
the elementary inequality
\begin{align*}
(a+b)^2 \le (1 + \epsilon^{-\frac{1}{2}}) a^2 + (1 + \epsilon^{\frac{1}{2}})
b^2\,, \quad \forall a, b \in \mathbb{R}\,,
\end{align*}
together with the Poincar\'e inequality for the reduced Markov chain
$\bar{\mathcal{C}}$. It gives 
\begin{align}
  \mbox{Var}_\pi f = & \sum_{x \in X} \Big(f(x) - \sum_{x' \in X}
  f(x')\pi(x')\Big)^2 \pi(x) \notag \\
  = & \sum_{x \in X} \Big(f(x) - \sum_{i=1}^m \widetilde{f}(i)w(i)\Big)^2 \pi(x) \notag\\
\le & (1 + \epsilon^{-\frac{1}{2}}) \sum_{i=1}^m\sum_{x\in X_i} \Big(f(x) -
  \widetilde{f}(i)\Big)^2
  \pi(x) + (1 + \epsilon^{\frac{1}{2}}) \sum_{i=1}^m \Big(\widetilde{f}(i) -
  \sum_{j=1}^m\widetilde{f}(j)w(j)\Big)^2 w(i) \notag\\
  \le & (1 +
  \epsilon^{-\frac{1}{2}})\frac{\bar{\lambda}\,\epsilon}{\lambda_{\min}}
  \mbox{Var}_\pi f + \frac{1}{2}(1 + \epsilon^{\frac{1}{2}}) \bar{\lambda}^{-1}
  \sum_{1 \le i,j \le m} \Big(\widetilde{f}(j) - \widetilde{f}(i)\Big)^2
 \bar{Q}(i,j) w(i)\,. \label{variance-ineq}
\end{align}
Combining (\ref{dirichlet-ineq})--(\ref{variance-ineq}), we arrive at
\begin{align*}
\mathcal{E}_\epsilon(f,f) \ge 
\frac{\bar{\lambda}}{(1 + \epsilon^{\frac{1}{2}})^2} \mbox{Var}_\pi f \Big[1 - (1 +
\epsilon^{-\frac{1}{2}})\frac{\bar{\lambda}\,\epsilon}{\lambda_{\min}}\Big]
- \epsilon^{\frac{1}{2}}
\frac{\bar{\lambda}Q_\infty}{\lambda_{\min}} \mbox{Var}_\pi f  \,,
\end{align*}
which implies  
\begin{align*}
\lambda_\epsilon \ge 
\frac{\bar{\lambda}}{(1 + \epsilon^{\frac{1}{2}})^2} \Big(1 - 
  \frac{2\bar{\lambda}\,\epsilon^{\frac{1}{2}}}{\lambda_{\min}}\Big)
- \frac{\bar{\lambda}Q_\infty}{\lambda_{\min}} \epsilon^{\frac{1}{2}} \,
\end{align*}
when $\epsilon \le 1$.
\end{enumerate}
\end{proof}

We continue to study the logarithmic Sobolev constant.
\begin{theorem}
  Assume Markov chain $\mathcal{C}$ is reversible. Let
  $\alpha_\epsilon, \bar{\alpha}$ be the logarithmic Sobolev 
  constants of Markov chain $\mathcal{C}$ and $\bar{\mathcal{C}}$ corresponding to infinitesimal generator $Q$ and
  $\bar{Q}$, respectively. We have 
\begin{align*}
  \frac{\bar{\alpha}}{1 + \epsilon^{\frac{1}{2}}}\Big(1 -
\frac{(\bar{\alpha} + \frac{1}{4} Q_\infty)\epsilon}{\alpha_{\min}} \Big) - 
  \frac{\bar{\alpha} Q_\infty \epsilon^{\frac{1}{2}}}{2\alpha_{\min}}
  \le \alpha_\epsilon \le \bar{\alpha}\,.
\end{align*}
\label{thm-2}
\end{theorem}
\begin{proof}
Recall the logarithmic Sobolev constant defined in (\ref{log-sobolev-const})
\begin{align}
  \alpha_\epsilon =
  \inf_{f}\Big\{\frac{\mathcal{E}_\epsilon(f,f)}{\mbox{Ent}_\pi(f^2)}\,~\Big|~
  \mbox{Ent}_\pi(f^2) > 0\,,\, f : X \rightarrow \mathbb{R}\Big\}\,.
  \label{log-sobolev-const-repeat}
\end{align}
\begin{enumerate}
  \item
  The upper bound follows directly by considering functions
  $f(x) = g(i)$ when $x \in X_i$, $g : \bar{X}\rightarrow \mathbb{R}$, and
  noticing that $\mbox{Ent}_\pi(f^2) = \mbox{Ent}_w(g^2)$,
  $\mathcal{E}_\epsilon(f,f) = \bar{\mathcal{E}}(g,g)$. 
\item 
For the lower bound, 
we assume the minimum in (\ref{log-sobolev-const-repeat}) is achieved with
function $f$, i.e.
\begin{align*}
  \alpha_\epsilon = \frac{\mathcal{E}_\epsilon(f,f)}{\mbox{Ent}_\pi (f^2)} \le \bar{\alpha}.
\end{align*} 
Estimation of $\alpha_\epsilon$ can be obtained if we could estimate both
the numerator and denominator.
For the Dirichlet form $\mathcal{E}_\epsilon(f,f)$, from (\ref{dirichlet-form-f-g-multiscale}) we have 
\begin{align*}
  \frac{1}{2\epsilon} \sum_{i = 1}^m \sum_{x, x' \in X_i} \big(f(x') - f(x)\big)^2
  Q_{0,i}(x,x') \pi(x) \le \bar{\alpha}\mbox{Ent}_\pi(f^2)\,.
\end{align*}
Applying Poincar\'e inequality and
logarithmic Sobolev inequality for Markov chain $\mathcal{C}_i$ for each fixed $i$, and noticing
$2\alpha_i \le \lambda_i$ (see (\ref{order-diff-const-reversible})), we obtain
\begin{align}
  \begin{split}
  & \sum_{i=1}^m \Big[\sum_{x \in X_i} \Big(f(x) -
  \widetilde{f}(i)\Big)^2\pi_i(x)\Big]w(i) 
  \le \frac{\bar{\alpha}}{2\alpha_{\min}} \mbox{Ent}_\pi(f^2) \,
  \epsilon\,, \\
  & 
  \sum_{i=1}^m \Big[\sum_{x \in X_i} |f(x)|^2 \ln\frac{|f(x)|^2}{F(i)^2}
  \pi_i(x)\Big] w(i) \le \frac{\bar{\alpha}}{\alpha_{\min}}
  \mbox{Ent}_\pi(f^2)\,\epsilon \,,
\end{split}
\label{poincare-log-sobolev}
\end{align}
respectively. In the above, $F(i)^2 = |f(i, \cdot)|^2_{2,\pi_i}=\sum\limits_{x \in X_i} |f(x)|^2 \pi_i(x)$ and we have
$|\widetilde{f}(i)| \le F(i)$, $1\le i \le m$. We proceed similarly as in the proof of Theorem~\ref{thm-1}. 
 Applying the detailed balance
condition (\ref{detail-balance-c}), inequalities (\ref{basic-ineq-1}) and (\ref{poincare-log-sobolev})
to the Dirichlet form (\ref{dirichlet-form-f-g-multiscale}), we can obtain
\begin{align}
  & \mathcal{E}_\epsilon(f,f)\\
  \ge & \frac{1}{2} 
  \sum_{1 \le i \neq j\le m}\sum_{x \in X_i, y \in X_j} \Big(f(y) - f(x)\Big)^2 Q_1(x,y)
  \pi(x) \notag \\
  \ge & \frac{1}{1 + \epsilon^{\frac{1}{2}}} \frac{1}{2} \sum_{1 \le i \neq j\le m} 
  \Big(\widetilde{f}(j) - \widetilde{f}(i)\Big)^2 \bar{Q}(i,j) w(i) 
  - 2\epsilon^{-\frac{1}{2}} \sum_{i=1}^m
  \sum_{x \in X_i, y \not\in X_i} \Big(f(x) - \widetilde{f}(i)\Big)^2 Q_1(x,y) \pi(x) \notag \\
  \ge & \frac{1}{1 + \epsilon^{\frac{1}{2}}} \frac{1}{2} \sum_{1 \le i \neq j\le m} 
  \Big(\widetilde{f}(j) - \widetilde{f}(i)\Big)^2 \bar{Q}(i,j) w(i) 
  - Q_\infty \epsilon^{-\frac{1}{2}} \sum_{i=1}^m
  \sum_{x \in X_i} \Big(f(x) - \widetilde{f}(i)\Big)^2 \pi(x) \notag \\
  \ge & \frac{1}{1 + \epsilon^{\frac{1}{2}}} \frac{1}{2} \sum_{1 \le i \neq j\le m} 
  \Big(\widetilde{f}(j) - \widetilde{f}(i)\Big)^2 \bar{Q}(i,j) w(i) 
  - \frac{\bar{\alpha}Q_\infty\mbox{Ent}_\pi(f^2)
  \epsilon^{\frac{1}{2}}}{2\alpha_{\min}}\,. 
\label{dirichlet-ineq-log}
\end{align}
Now we estimate $\mbox{Ent}_\pi(f^2)$. Using
(\ref{poincare-log-sobolev}), the logarithmic Sobolev
inequality for the reduced Markov chain $\bar{\mathcal{C}}$, and 
noticing $|f|^2_{2,\pi} = |F|^2_{2,w}$, we have 
\begin{align}
  \mbox{Ent}_\pi(f^2) = & \sum_{i=1}^m\sum_{x \in X_i} |f(x)|^2 \ln \frac{|f(x)|^2}{|F(i)|^2}
  \pi_i(x) w(i)
  + \sum_{i=1}^m |F(i)|^2 \ln \frac{|F(i)|^2}{|f|^2_{2,\pi}} w(i) \notag \\
  \le & \frac{\bar{\alpha}}{\alpha_{\min}} \mbox{Ent}_\pi(f^2)\,\epsilon + \frac{1}{2\bar{\alpha}} 
  \sum_{1 \le i,j \le m} |F(j) - F(i)|^2 \bar{Q}(i,j) w(i)\,.
  \label{l-entropy-term}
\end{align}
The first inequality in (\ref{poincare-log-sobolev}) and the
definition in (\ref{q-bar}) imply
\begin{align*}
  & 0 \le \sum_{i = 1}^m \Big(F(i)^2 - \widetilde{f}(i)^2\Big) w(i) \le 
  \frac{\bar{\alpha}}{2\alpha_{\min}}  \mbox{Ent}_\pi(f^2) \, \epsilon \,, \\
& 0 < \sum_{j\neq i} \bar{Q}(i,j) = \sum_{x \in X_i}\sum_{y \not\in X_i} Q_1(x,y)
\pi_i(x) \le \frac{Q_\infty}{2}\,.
\end{align*}
Then the second term on the right hand side of (\ref{l-entropy-term}) can be
bounded as 
\begin{align*}
  & \sum_{1 \le i,j \le m} \Big|F(j) - F(i)\Big|^2 \bar{Q}(i,j) w(i) \\
  =& \sum_{1 \le i\neq j \le m} \Big(F(i)^2 - 2F(i)F(j) + F(j)^2\Big)
  \bar{Q}(i,j)w(i) \\
  \le & 2 \sum_{1 \le i\neq j\le m} F(i)^2 \bar{Q}(i,j) w(i) - 2
  \sum_{1 \le i\neq j\le m}\widetilde{f}(i)\widetilde{f}(j) \bar{Q}(i,j) w(i) \\
  = & \sum_{1 \le i,j \le m} \Big(\widetilde{f}(j) - \widetilde{f}(i)\Big)^2 \bar{Q}(i,j)
  w(i) + 2 \sum_{1 \le i\neq j \le m} \Big(F(i)^2 - \widetilde{f}(i)^2\Big) \bar{Q}(i,j) w(i)
  \\
\le &\sum_{1 \le i,j \le m} \Big(\widetilde{f}(j) - \widetilde{f}(i)\Big)^2 \bar{Q}(i,j)
  w(i) + 
  \frac{Q_\infty\bar{\alpha}}{2\alpha_{\min}} \mbox{Ent}_\pi(f^2) \, \epsilon \,,
\end{align*}
where the detailed balance condition (\ref{detail-balance-c-bar}) for Markov
chain $\bar{\mathcal{C}}$ has been used. 

Therefore (\ref{l-entropy-term}) implies
\begin{align*}
  \mbox{Ent}_\pi(f^2) \le 
 \frac{\bar{\alpha}}{\alpha_{\min}}
 \mbox{Ent}_\pi(f^2)\,\epsilon + \frac{1}{2\bar{\alpha}} \Big[
\sum_{1 \le i,j \le m} \Big(\widetilde{f}(j) - \widetilde{f}(i)\Big)^2 \bar{Q}(i,j)
w(i) + \frac{Q_\infty\bar{\alpha}}{2\alpha_{\min}}  \mbox{Ent}_\pi(f^2) \,
\epsilon  \Big]\,,
\end{align*}
or equivalently
\begin{align*}
  \frac{1}{2} \sum_{1 \le i,j \le m} \Big(\widetilde{f}(j) - \widetilde{f}(i)\Big)^2 \bar{Q}(i,j) w(i) 
  \ge \bar{\alpha} \mbox{Ent}_\pi(f^2) \Big(1 - \frac{(\bar{\alpha} +
  \frac{1}{4}Q_\infty)\epsilon}{\alpha_{\min}} \,\Big)\,.
\end{align*}
Substituting the above inequality into (\ref{dirichlet-ineq-log}), we obtain
\begin{align*}
  \mathcal{E}_\epsilon(f,f) \ge \frac{\bar{\alpha}\mbox{Ent}_\pi(f^2)}{1+\epsilon^{\frac{1}{2}}}
  \Big(1 - 
  \frac{(\bar{\alpha} + \frac{1}{4} Q_\infty)\epsilon}{\alpha_{\min}} 
  \Big) - 
 \frac{\bar{\alpha}Q_\infty \mbox{Ent}_\pi(f^2)\epsilon^{\frac{1}{2}}
}{2\alpha_{\min}}\,,
\end{align*}
which indicates the lower bound. 
\end{enumerate}
\end{proof}

Finally, we study the modified logarithmic Sobolev constant.
\begin{theorem} Let $\gamma_\epsilon, \bar{\gamma}$ be the modified
  logarithmic Sobolev 
  constants of the reversible Markov chain $\mathcal{C}$ and $\bar{\mathcal{C}}$ corresponding to infinitesimal generator $Q$ and
  $\bar{Q}$, respectively. We have 
  \begin{align}
  \lim_{\epsilon \rightarrow 0} \gamma_{\epsilon} = \bar{\gamma}\,.
\end{align}
\label{thm-3}
\end{theorem}
\begin{proof}
We argue by contradiction. Suppose the conclusion is not true. First 
recall the modified logarithmic Sobolev constant defined in (\ref{mlsi-def})
\begin{align}
  \gamma_\epsilon = \inf_{f} \Big\{ \frac{\mathcal{E}_\epsilon(f, \ln
  f)}{\mbox{Ent}_{\pi}(f)}~\Big|~ \mbox{Ent}_{\pi}(f) > 0\,, f : X \rightarrow \mathbb{R}^+ \Big\}\,.
    \label{mlsi-def-repeat}
\end{align}
Take functions
$f(x) = g(i)$ for $x \in X_i$,  then it is straightforward to check
$\mathcal{E}_\epsilon(f, \ln f) = \bar{\mathcal{E}}(g,\ln g)$ and
$\mbox{Ent}_\pi(f) = \mbox{Ent}_w(g)$, therefore we can deduce
$\gamma_\epsilon \le
\bar{\gamma}$ by allowing function $g$ to vary among all functions $g : \bar{X}
\rightarrow \mathbb{R}^+$.

 Since $\gamma_\epsilon \le \bar{\gamma}$, we can find a sequence $\epsilon^{(k)}$,
$\lim\limits_{k\rightarrow +\infty} \epsilon^{(k)} = 0$, s.t.
$\lim\limits_{k\rightarrow +\infty} \gamma^{(k)} < \bar{\gamma}$.
Notice that in this proof, we will use notations $\gamma^{(k)}$,
$\mathcal{E}^{(k)}$ instead
of $\gamma_{\epsilon^{(k)}}$ and $\mathcal{E}_{\epsilon^{(k)}}$.
We assume the infima in (\ref{mlsi-def-repeat}) are achieved with functions
$f_k: X \rightarrow \mathbb{R}^+$, i.e.
\begin{align}
  \gamma^{(k)}= \frac{\mathcal{E}^{(k)}(f_k, \ln
  f_k)}{\mbox{Ent}_{\pi} (f_k)}\,, \quad \mathbf{E}_{\pi}
  f_k= \sum_{x \in X} f_k(x) \pi(x) = 1\,.
  \label{mlsi-minimizer}
\end{align}
Let $\pi_{\min} = \min\limits_{x \in X} \pi(x)$. Clearly we have $0 < f_k(x) \le \pi^{-1}_{\min}$, $\forall x \in
X$ (see \cite{mlsi-bobkov} for the positivity), and therefore
\begin{align}
  0 \le \mbox{Ent}_\pi(f_k) \le \ln \frac{1}{\pi_{\min}}\,.
\end{align}
Since $f_k$ are bounded, we further assume they converge to some function
$\bar{f} : X \rightarrow \mathbb{R}$ for each $x \in X$ (This can be achieved
by considering a convergent subsequence).

From Dirichlet form (\ref{dirichlet-form-f-g-multiscale}) and (\ref{mlsi-minimizer}), we have
\begin{align}
  \begin{split}
    & \frac{1}{\epsilon^{(k)}} \sum_{i=1}^m \mathcal{E}_i(f_k(i, \cdot), \ln
    f_k(i, \cdot)) w(i) \\
  = & \frac{1}{2\epsilon^{(k)}} \sum_{i=1}^m\sum_{x,x' \in X_i} (f_k(x') -
  f_k(x)) (\ln
f_k(x') - \ln f_k(x)) Q_{0, i}(x,x') \pi(x) \\
\le & \mathcal{E}^{(k)}(f_k, \ln f_k) = \gamma^{(k)} \mbox{Ent}_\pi (f_k)\, \le 
\bar{\gamma} \ln\frac{1}{\pi_{\min}} \,.
\end{split}
  \label{mlsi-term-1}
\end{align}
Recall $\gamma_i > 0$ is the modified logarithmic Sobolev constant of Markov chain
$\mathcal{C}_i$, together with (\ref{mlsi-term-1}), we can obtain
\begin{align}
  \sum_{i=1}^m \mbox{Ent}_{\pi_i}(f_k(i, \cdot))  w(i) \le
  \frac{\bar{\gamma} \ln \frac{1}{\pi_{\min}} \,
  \epsilon^{(k)}}{\gamma_{\min}}\,.
\end{align}
Let $\mu_i^{(k)}$ be the probability measure on $X_i$ s.t.
$\mu_i^{(k)}(x) = \frac{f_{k}(x)}{\widetilde{f}_{k}(i)}\pi_i(x)$, $\forall x \in
X_i$. Applying Csisz\'{a}r-Kullback-Pinsker inequality, we can obtain
\begin{align}
  & \sum_{i=1}^m \Big[\sum_{x \in X_i}
  |f_k(x) - \widetilde{f}_k(i)| \pi_i(x)\Big] w(i) 
  =  \sum_{i=1}^m \widetilde{f}_k(i) ||\mu_i^{(k)}-
\pi_i||_{\mbox{\tiny{TV}}}\, w(i) \notag \\
\le &  \sum_{i=1}^m \widetilde{f}_k(i) \sqrt{2 \mbox{Ent}_{\pi_i}
\Big(\frac{f_k(i, \cdot)}{\widetilde{f}_k(i)}\Big)} w(i) 
=  \sum_{i=1}^m  \sqrt{2 \widetilde{f}_k(i)\mbox{Ent}_{\pi_i}
\big(f_k(i, \cdot)}\big) w(i) \notag \\
\le& \sqrt{2}  \Big[\sum_{i=1}^m \mbox{Ent}_{\pi_i} \big(f_k(i, \cdot)\big) w(i)
\Big]^{\frac{1}{2}} \le \Big(
  \frac{2\bar{\gamma} \ln \frac{1}{\pi_{\min}} \,
  \epsilon^{(k)}}{\gamma_{\min}}\,\Big)^{\frac{1}{2}}\,,\label{tv-bound}
\end{align}
where $\|\cdot\|_{\mbox{\tiny{TV}}}$ denotes the total variation distance of
two probability measures, and we have used the fact 
\begin{align}
  \sum_{i=1}^m \widetilde{f}_k(i) w(i) = \sum_{x \in X} f_k(x)
  \pi(x) = 1\,.
\end{align}
Taking the limit $k\rightarrow +\infty$, inequality (\ref{tv-bound}) indicates that 
$\bar{f}$ is constant on each subset $X_i$, i.e. 
$\bar{f}(x) = \bar{g}(i)$ if $x \in X_i$, where $\bar{g} : \bar{X} \rightarrow \mathbb{R}^+$. 
We argue that $\bar{f}$ is both positive and non-constant (this argument is adapted from \cite{mlsi-bobkov}). Suppose $\bar{f}$
is constant. Notice that 
\begin{align}
  \mathbf{E}_\pi \bar{f} = \mathbf{E}_w \bar{g} = \lim_{k \rightarrow +
  \infty} \mathbf{E}_\pi f_k = 1\,,
\end{align}
therefore we must have $\bar{f} \equiv 1$. Let $f_k= 1 + f'_k$, where
$\mathbf{E}_\pi(f'_k) = 0$ and  
$\lim\limits_{k\rightarrow +\infty} f'_k(x) = 0$, $\forall x \in X$.
Using Taylor expansion, we can verify 
\begin{align*}
  &\mathcal{E}^{(k)}(f_k, \ln f_k) = \Big(1 + \mathcal{O}(|f'_k|_{\infty})\Big)
  \mathcal{E}^{(k)}(f'_k, f'_k)\,, \\
  &\mbox{Ent}_{\pi} f_k = \frac{1}{2} \mbox{Var}_\pi f'_k +
  \mathcal{O}(|f'_k|^3_{\infty})\,.
\end{align*}
Because $\mbox{Var}_\pi f'_k \ge |f'_k|^2_{\infty} \pi_{\min}$, we know 
$\mbox{Ent}_{\pi} f_k = \frac{1 + o(1)}{2} \mbox{Var}_\pi f'_k$. Applying 
Poincar\'e inequality and Theorem~\ref{thm-1} (which implies $\lambda^{(k)}$ converges to
$\bar{\lambda}$), we can estimate
\begin{align}
  \bar{\gamma} > \lim_{k\rightarrow +\infty} \gamma^{(k)} \ge \liminf_{k \rightarrow
  \infty} 
  \frac{\mathcal{E}^{(k)}(f'_k, f'_k)}{\frac{1}{2}\mbox{Var}_{\pi}(f'_k)}
  \ge 2 \lim_{k\rightarrow +\infty} \lambda^{(k)} = 2\bar{\lambda}\,,
\end{align}
which contradicts to the relation (\ref{order-diff-const-bar-reversible}). Therefore function $\bar{f}$
is non-constant. Now we consider the positiveness. Define two disjoint sets 
\begin{align}
  M = \Big\{x \in X ~|~ \bar{f}(x) = 0\Big\}\,, \quad M' = \Big\{x \in X ~|~
  \bar{f}(x) > 0\Big\}\,.
\end{align}
Since $\mathbf{E}_\pi \bar{f} = 1$, we know $M'$ is not empty. Now assume set $M$ is also nonempty. 
Applying the irreducibility of Markov chain $\mathcal{C}$ to subset $M$ and $M'$, we
conclude that $\exists x \in M$, $y \in M'$  s.t. $Q(x,y) > 0$. 
Since $\bar{f}$ is constant on each subset $X_i$, we know $x \in X_i, y \in
X_j$, for some $i \neq j$ and $Q_1(x,y) > 0$. Then we have 
\begin{align*}
  +\infty =& \lim_{k\rightarrow +\infty} (f_k(y) - f_k(x)) (\ln f_k(y) - \ln
  f_k(x)) Q_1(x,y) \pi(x)  \\
  \le& \limsup_{k\rightarrow +\infty} \mathcal{E}^{(k)}(f_k, \ln f_k)  
  = \limsup_{k\rightarrow +\infty}\gamma^{(k)} \mbox{Ent}_\pi(f_k) \le
  \bar{\gamma}  \ln \frac{1}{\pi_{\min}}\,.
\end{align*}
This contradiction shows that $M$ is empty and therefore $\bar{f}$ is
positive. Now we can take the limits 
\begin{align*}
  &  \liminf_{k \rightarrow +\infty} \mathcal{E}^{(k)}(f_k, \ln f_k) \ge
 \bar{\mathcal{E}}(\bar{g}, \ln\bar{g})\,
 \\
 &
\lim_{k \rightarrow +\infty} \mbox{Ent}_\pi(f_k) =
\mbox{Ent}_\pi(\bar{f}) = \mbox{Ent}_w(\bar{g})\,, 
\end{align*}
and 
\begin{align}
  \bar{\gamma} > \lim_{k\rightarrow + \infty} \gamma^{(k)} = \lim_{k\rightarrow +
  \infty}
  \frac{\mathcal{E}^{(k)}(f_k, \ln f_k)}{\mbox{Ent}_\pi(f_k)}  \ge
  \frac{\bar{\mathcal{E}}(\bar{g}, \ln \bar{g})}{\mbox{Ent}_w(\bar{g})} \,.
\end{align}
But the above inequality is in contradiction with the fact that
$\bar{\gamma}$ is the modified logarithmic Sobolev constant of the reduced Markov chain
$\bar{C}$.  Therefore we conclude $\lim\limits_{\epsilon \rightarrow
0} \gamma_\epsilon = \bar{\gamma}$.
\end{proof}
\begin{remark}
  Theorem~\ref{thm-1} and Theorem~\ref{thm-2} imply that 
  both the Poincar\'e constant $\lambda_\epsilon$ and logarithmic Sobolev constant
  $\alpha_\epsilon$ of Markov chain
  $\mathcal{C}$ converge to their counterparts $\bar{\lambda}$, $\bar{\alpha}$ of the reduced Markov chain
  $\bar{\mathcal{C}}$ and the convergence order is
  $\mathcal{O}(\epsilon^{\frac{1}{2}})$.

  On the other hand, the result of Theorem~\ref{thm-3} is weaker in that we
  obtain convergence without convergence order.
\end{remark}

\section{Asymptotic analysis : general case}
\label{sec-general}
In this section we consider the general case without assuming reversibility. 
A convergence result of Kolmogorov backward equation can be found in
\cite{pavliotis2008multiscale} and will not be discussed here (also see
Appendix~\ref{app-sec-2}).
Unlike the reversible case in Section~\ref{sec-reversible}, 
relation (\ref{pi-w}) does not hold in general and the invariant measure $\pi^\epsilon$ will
depend on parameter $\epsilon$. Instead, we have the following result (recall the notations in
Subsection~\ref{subsec-notation}). 
\begin{theorem}
  Let $\pi^\epsilon$, $w$, $\pi_i$ be the invariant
  measures of Markov chain
  $\mathcal{C}$, $\bar{\mathcal{C}}$, and $\mathcal{C}_i$, respectively.
  We have 
\begin{align*}
   \Big(\sum_{i=1}^m\sum_{x \in X_i} |\pi^\epsilon(x) - w(i)\pi_i(x)|^2\Big)^{\frac{1}{2}} \
  \le 
 3^{\frac{1}{2}}\Big[
\frac{1}{2} + \Big(\frac{1}{2} + m\Big) \bar{\sigma}^{-2} d\,\Gamma
+ n \Big]^{\frac{1}{2}} \sigma_{\min}^{-1}Q_\infty \epsilon\,.
\end{align*}
\label{thm-pi-general}
\end{theorem}
\begin{proof}
We first
study the
invariant measure $\pi^\epsilon$, which satisfies equation
\begin{align}
  (\frac{Q_0^T}{\epsilon} + Q_1^T) \pi^\epsilon = 0 \,,
  \label{multi-invariant-eqn}
\end{align}
i.e. $Q_0^T \pi^\epsilon = - \epsilon Q_1^T \pi^\epsilon$. Since matrix $Q_0$ is a block
diagonal matrix
given in (\ref{q0-block}),  we obtain linear systems  
\begin{align}
  Q_{0,i}^T\pi^\epsilon(i, \cdot) = - \epsilon b(i, \cdot)\,, \quad 1\le i \le m\,,
  \label{multi-invariant-eqn-1}
\end{align}
where $b(i,x) = \sum\limits_{y \in X} Q_1(y,x)\pi^\epsilon(y)$, $x\in X_i$. 
Summing up $x \in X_i$ in (\ref{multi-invariant-eqn-1}) and
noticing that each matrix $Q_{0,i}$ have zero row sums, 
we obtain 
\begin{align}
\sum\limits_{x \in X_i} b(i,x) = \sum_{x \in X_i}\sum_{y \in X}
Q_1(y,x)\pi^\epsilon(y) = 0\,,\quad 1 \le i \le m\,. 
\label{sum-zero-general}
\end{align}
Using (\ref{Q-inf}), we can estimate 
\begin{align*}
   |b(i,\cdot)|_2   
  =& \Big[\sum_{x \in X_i}\Big(\sum_{y \in X} Q_1(y,x)
\pi^\epsilon(y)\Big)^2\Big]^{\frac{1}{2}} 
\le \Big[\sum_{x \in X_i}\sum_{y \in X} Q^2_1(y,x)
\pi^\epsilon(y)\Big]^{\frac{1}{2}} \\
\le & \Big[\max_{x,y \in X} |Q_1(x,y)| \sum_{x \in X_i}\sum_{y \in X }
|Q_1(y,x)|\pi^\epsilon(y) \Big]^{\frac{1}{2}} \\
\le& \Big[\frac{1}{2} Q_\infty\sum\limits_{x \in X_i}\sum\limits_{y \in X }
|Q_1(y,x)|\pi^\epsilon(y)
\Big]^{\frac{1}{2}} =: g(i) \,,
\end{align*}
which implies
\begin{align}
  \sum\limits_{i=1}^m |b(i)|_2^2\le \sum_{i=1}^m g(i)^2 \le
  \frac{1}{2}Q_{\infty}^2 \,.
  \label{b-l2-2}
\end{align}
Applying Lemma~\ref{lem-1} to equation (\ref{multi-invariant-eqn-1}), we have
\begin{align}
  |\pi^\epsilon(i,\cdot) - w^\epsilon(i)\pi_i|_\infty \le
  |\pi^\epsilon(i,\cdot) - w^\epsilon(i)\pi_i|_2 \le
  \epsilon\sigma_{i}^{-1} g(i)
  \label{general-1}
\end{align}
for some $w^\epsilon(i) \in \mathbb{R}$. Recall that $\sigma_{i}$ is the smallest nonzero
singular value of matrix $Q_{0,i}$. 
Let $\pi^\epsilon(i,\cdot) = w^\epsilon(i)\pi_i(\cdot) + r^\epsilon(i,\cdot)$, $1 \le i \le m$, using
(\ref{sum-zero-general}), we have 
\begin{align}
  \begin{split}
   \sum_{j=1}^m \bar{Q}(j,i) w^\epsilon(j) 
  = &
  \sum_{x \in X_i} \sum_{j=1}^m \sum_{y \in X_j} Q_1(y,x) \pi_i(y)w^\epsilon(j)\\
  =&
  - \sum_{x \in X_i} \sum_{j=1}^m \sum_{y \in X_j} Q_1(y,x)r^\epsilon(j,y) =:
  \bar{b}^\epsilon(i)\,.
\end{split}
  \label{q-bar-w}
\end{align}
We also have 
\begin{align}
  |\bar{b}^\epsilon|_2 &= \Big[\sum_{i=1}^m\Big|\sum_{x \in X_i}\sum_{j=1}^m\sum_{y \in X_j}
Q_1(y,x)r^\epsilon(j,y)\Big|^2\Big]^{\frac{1}{2}} \notag \\
&\le 
\Big[\sum_{i=1}^m\Big(\sum_{x \in X_i} \sum_{y \in X}
Q_1(y,x)^2\Big)\Big( \sum_{j=1}^m\sum_{y \in X_j} \sum_{x \in X_i, y\sim x}
r^\epsilon(j,y)^2\Big)\Big]^{\frac{1}{2}} \,,
\label{b-bar-2}
\end{align}
where $y\sim x$ means $Q_1(y,x) \neq 0$.
From (\ref{b-l2-2}) and (\ref{general-1}), 
\begin{align*}
  & \Big(\sum_{j=1}^m\sum_{y \in X_j} \sum_{x \in X_i,y \sim x}
  r^\epsilon(j,y)^2\Big)^{\frac{1}{2}}
  \le \Big(d \sum_{j=1}^m |r^\epsilon(j,\cdot)|_2^2\Big)^{\frac{1}{2}} \le
  \Big(\frac{d}{2}\Big)^{\frac{1}{2}} 
  \sigma_{\min}^{-1} Q_\infty \epsilon\,,
\end{align*}
therefore (\ref{b-bar-2}) becomes 
\begin{align*}
|\bar{b}^\epsilon|_2 \le \epsilon \sigma_{\min}^{-1} Q_{\infty} \Big(\frac{d\,\Gamma}{2}\Big)^{\frac{1}{2}} \,,
\end{align*}
where $\Gamma = \mbox{tr} (Q_1Q_1^T)$.
Now applying Lemma~\ref{lem-1} to equation (\ref{q-bar-w}), we obtain
\begin{align}
  |w^\epsilon - \lambda w|_\infty \le |w^\epsilon - \lambda w|_2 \le
(\bar{\sigma}\sigma_{\min})^{-1} 
Q_\infty \Big(\frac{d\,\Gamma}{2}\Big)^{\frac{1}{2}} 
\epsilon\,,
\label{general-2}
\end{align}
where $\lambda \in \mathbb{R}$, $w$ is the invariant measure
of the reduced Markov chain $\bar{\mathcal{C}}$, i.e. $\bar{Q}^T w = 0$ and
$\bar{\sigma}$ is the smallest nonzero singular value of
$\bar{Q}$. Therefore 
\begin{align}
  & \Big|\sum_{i=1}^m \big(w^\epsilon(i) - \lambda w(i)\big)\Big| \le
m^{\frac{1}{2}} |w^\epsilon-\lambda w|_2 \le 
(\bar{\sigma}\sigma_{\min})^{-1} Q_\infty \Big(\frac{m\,d\,\Gamma}{2}\Big)^{\frac{1}{2}} 
\epsilon\,. 
\label{diff-w-lambda-wbar}
\end{align}
From (\ref{b-l2-2}), (\ref{general-1}) and 
\begin{align*}
  \sum_{i=1}^m w^\epsilon(i) = 1 - \sum_{i=1}^m\sum_{x\in X_i} r^\epsilon(i,x)\,,\quad \sum_{i=1}^m w(i) =
  1\,,
\end{align*}
we can estimate
\begin{align*}
  & \Big|\sum_{i=1}^m \big(w^\epsilon(i) - \lambda w(i)\big)\Big| = \Big|1 - \lambda -
  \sum_{i=1}^m\sum_{x \in X_i} r^\epsilon(i,x)\Big| \ge |1 - \lambda| -
  \Big|\sum_{i=1}^m\sum_{x \in X_i} r^\epsilon(i,x)\Big|\,, \\
& \Big|\sum_{i=1}^m\sum_{x\in X_i} r^\epsilon(i,x)\Big| \le \sum_{i=1}^m \big|\sum_{x\in
X_i} r^\epsilon(i,x)\big| \le 
\sum_{i=1}^m n_{i}^{\frac{1}{2}} \sigma_{i}^{-1} g(i)\epsilon
\le  \Big(\frac{n}{2}\Big)^{\frac{1}{2}} \sigma_{\min}^{-1} Q_\infty\, \epsilon\,.
\end{align*}
Together with (\ref{diff-w-lambda-wbar}), we know 
\begin{align}
  |1 - \lambda| \le \Big[\bar{\sigma}^{-1} 
    \Big(\frac{m\,d\,\Gamma}{2}\Big)^{\frac{1}{2}} + 
  \Big(\frac{n}{2}\Big)^\frac{1}{2} \Big] \sigma_{\min}^{-1}Q_\infty \epsilon\,.
\label{general-lambda}
\end{align}
Combining (\ref{general-1}), (\ref{general-2}) and (\ref{general-lambda}), we have 
\begin{align*}
  & \Big(\sum_{i=1}^m\sum_{x \in X_i} |\pi^\epsilon(x) - w(i)\pi_i(x)|^2\Big)^{\frac{1}{2}} \\
  \le &3^{\frac{1}{2}}\Big[\sum_{i=1}^m |\pi^\epsilon(i,\cdot) - w^\epsilon(i) \pi_i|^2_2 +
  |w^\epsilon - \lambda w|^2_2 +|\lambda-1|^2 \sum_{i=1}^m w(i)^2\Big]^{\frac{1}{2}} \\
\le & 3^{\frac{1}{2}}\Big\{\frac{1}{2} + \frac{\bar{\sigma}^{-2} d\,\Gamma}{2}
+ \Big[\bar{\sigma}^{-1} 
  \Big(\frac{m\,d\,\Gamma}{2}\Big)^{\frac{1}{2}} + 
\Big(\frac{n}{2}\Big)^\frac{1}{2} \Big]^2
\Big\}^{\frac{1}{2}} \sigma_{\min}^{-1}Q_\infty
\epsilon\,\\
  \le & 3^{\frac{1}{2}}\Big[
\frac{1}{2} + \Big(\frac{1}{2} + m\Big) \bar{\sigma}^{-2} d\,\Gamma
+ n \Big]^{\frac{1}{2}} \sigma_{\min}^{-1}Q_\infty \epsilon\,.
\end{align*}
\end{proof}

Based on Theorem~\ref{thm-pi-general}, we can obtain convergence results 
of various constants of Markov chain $\mathcal{C}$. In the proof of the following
result, we will use the fact that the infima in definitions
(\ref{poincare-const}), (\ref{log-sobolev-const}) and (\ref{mlsi-def}) can be attained by some
functions. This fact can be verified using arguments in \cite{mlsi-bobkov},
which is also valid in non-reversible case.
\begin{theorem}
  Let $\lambda_\epsilon$, $\alpha_\epsilon$, $\gamma_\epsilon$ be the
  Poincar\'e
  constant, logarithmic Sobolev constant and modified logarithmic Sobolev constant of Markov
  chain $\mathcal{C}$. Also let $\bar{\lambda}$, $\bar{\alpha}$, $\bar{\gamma}$ be
  their counterparts of Markov chain $\bar{\mathcal{C}}$. We have 
  \begin{align}
    \lim_{\epsilon \rightarrow 0} \lambda_\epsilon = \bar{\lambda}\,, \quad 
    \lim_{\epsilon \rightarrow 0} \alpha_\epsilon = \bar{\alpha}\,, \quad 
    \lim_{\epsilon \rightarrow 0} \gamma_\epsilon = \bar{\gamma}\,. 
  \end{align}
  \label{thm-constants-general}
\end{theorem}
\begin{proof}
  We will sketch the proof, since the argument is similar to
  Theorem~\ref{thm-3}.
  \begin{enumerate}
    \item
      First consider the Poincar\'e constant.
      Let function $g : \bar{X} \rightarrow \mathbb{R}$ satisfy $\mathbf{E}_w
      g = 0$ and $\mbox{Var}_w g = 1$. Define $f(x) =
      g(i)$ for $x \in X_i$. We know $|f|_\infty=|g|_\infty$ is bounded.
      Applying Theorem~\ref{thm-pi-general} and using (\ref{dirichlet-form-f-f-multiscale}), we know 
      \begin{align}
	&\mathcal{E}_\epsilon(f, f) = \frac{1}{2} \sum_{i \neq j}\sum_{x \in X_i}\sum_{y \in X_j} (f(y) - f(x))^2 Q_1(x,y) \pi^\epsilon(x) 
	=\bar{\mathcal{E}}(g,g) + \mathcal{O}(\epsilon)\,,\\
	&\mbox{Var}_{\pi^\epsilon} f = \sum_{i=1}^m g^2(i) \Big(\sum_{x \in X_i}
	\pi^\epsilon(x)\Big)  - \Big(\sum_{i=1}^m g(i) \sum_{x \in X_i}
	\pi^\epsilon(x)\Big)^2 = \mbox{Var}_w g + \mathcal{O}(\epsilon)\,.
      \end{align}
      Therefore 
      $$\limsup_{\epsilon \rightarrow 0} \lambda_\epsilon \le
      \limsup_{\epsilon \rightarrow 0} 
      \frac{\mathcal{E}_\epsilon(f, f)}{\mbox{Var}_{\pi^\epsilon} f} = 
      \frac{\bar{\mathcal{E}}(g, g)}{\mbox{Var}_{w} g}, $$
      and we obtain $\limsup\limits_{\epsilon \rightarrow 0} \lambda_\epsilon \le
      \bar{\lambda}$ after taking infimum among functions $g : \bar{X}
      \rightarrow \mathbb{R}$. 

      Now suppose the conclusion is not true, then we can find a sequence
      $\epsilon^{(k)}$, $\lim\limits_{k \rightarrow + \infty} \epsilon^{(k)} = 0$ and
      $\lim\limits_{k \rightarrow +\infty} \lambda^{(k)} < \bar{\lambda}$ (We
      use the same notations as in Theorem~\ref{thm-3}). Let
      function $f_k$ be the extreme functions in (\ref{poincare-const}) and satisfy $\mbox{Var}_{\pi^{(k)}} f_k = 1$ and
      $\mathbf{E}_{\pi^{(k)}} f_k = 0$. Then $\lambda^{(k)} =
      \mathcal{E}^{(k)}(f_k,f_k)$. It is easy to see 
      $$\limsup_{k\rightarrow
      +\infty} |f_k|_\infty < +\infty\,,$$ and therefore we can find a subsequence (also
      denoted as $f_k$ for simplicity) s.t. $f_k$ converges to $f : X
      \rightarrow \mathbb{R}$. Using $\lim\limits_{k\rightarrow +\infty}
      \mathcal{E}^{(k)}(f_k,f_k) \le \bar{\lambda}$, 
      we can deduce that $f(x) = g(i)$ if $x \in X_i$, for some $g :
      \bar{X} \rightarrow \mathbb{R}$. And $\mbox{Var}_w g = 1$, $\mathbf{E}_w
      g = 0$. Therefore 
      \begin{align}
	\bar{\lambda} \le \bar{\mathcal{E}}(g,g) \le \lim_{k \rightarrow +\infty}
	\mathcal{E}^{(k)} (f_k,f_k) = \lim_{k\rightarrow +\infty}
	\lambda^{(k)} < \bar{\lambda}\,.
      \end{align}
      This contradiction shows that $\lim\limits_{\epsilon \rightarrow 0} \lambda_\epsilon = \bar{\lambda}$\,.
    \item
      We continue to prove the convergence of the modified logarithmic Sobolev constant $\gamma_\epsilon$
      (the proof for the convergence of $\alpha_\epsilon$ is similar and is
      omitted). First of all, using a similar argument as above, we can obtain
      \begin{align}
	\limsup\limits_{\epsilon \rightarrow 0} \gamma_\epsilon \le
	\bar{\gamma}\,.
      \end{align}
      Suppose the conclusion is not true and then we can find sequence
      $\epsilon^{(k)}$, s.t. $\lim\limits_{k\rightarrow +\infty} \epsilon^{(k)} = 0$
      and $\lim\limits_{k \rightarrow +\infty} \gamma^{(k)}< \bar{\gamma}$.
      Let $f_k$ be the extreme functions in (\ref{mlsi-def}) with $\mathbf{E}_{\pi^{(k)}} f_k = 1$.
      We can argue as above that $\limsup\limits_{k \rightarrow +\infty}
      |f_k|_\infty < +\infty$, and there is a subsequence (also
      denoted as $f_k$) converging to some function $f$. Using
      (\ref{dirichlet-form-f-f-multiscale}) and Lemma~$2.7$ in \cite{log-sob-diaconis}, we have 
      \begin{align}
	&\frac{1}{2\epsilon^{(\epsilon)}} \sum_{i=1}^m \sum_{x,x' \in X_i}
	(f_k^{\frac{1}{2}}(x') - f_k^{\frac{1}{2}} (x))^2 Q_{0,i}(x,x') \pi^{(k)}(x) \notag \\
 \le & \mathcal{E}^{(k)}(f_k^{\frac{1}{2}}, f_k^{\frac{1}{2}}) \le \frac{1}{2} 
      \mathcal{E}^{(k)}(f_k, \ln f_k) =\frac{\lambda^{(k)}}{2}
      \mbox{Ent}_{\pi^{(k)}} (f_k)\,.
      \end{align}
      Taking limit $k \rightarrow + \infty$, applying
      Theorem~\ref{thm-pi-general} and the boundness of
      $\mbox{Ent}_{\pi^{(k)}}(f_k)$, we can deduce that $f$ is constant on each
      subset $X_i$, i.e. $f(x) = g(i)$ if $x \in X_i$, for some $g : \bar{X}
      \rightarrow \mathbb{R}^+$. We have $\mathbf{E}_w g =
      \lim\limits_{k\rightarrow +\infty} \mathbf{E}_{\pi^{(k)}} f_k = 1$. 
      The same argument as in Theorem~\ref{thm-3} shows that $g$ is positive. 
      Now we show $g$ is non-constant. Assume $g$ is constant and let $f_k =
      1 + f_k'$, then we have $\mathbf{E}_{\pi^{(k)}} f_k' = 0$ and $\lim\limits_{k\rightarrow +\infty} f_k'(x) = 0$, $\forall x \in X$. 
     Using Taylor expansion, we have 
     \begin{align*}
       \mathcal{E}^{(k)}(f_k, \ln f_k) 
       =&-\sum_{x \in X} f_k(x) \Big[\sum_{y \in X, y \neq x} Q(x,y) (\ln f_k(y) -
       \ln f_k(x))\Big] \pi^{(k)}(x) \\
       =& \big(1+\mathcal{O}(|f_k'|_\infty)\big) \mathcal{E}^{(k)} (f_k', f_k')
     \end{align*}
     and $\mbox{Ent}_{\pi^{(k)}}(f_k) = (\frac{1 + o(1)}{2})
     \mbox{Var}_{\pi^{(k)}} f_k'$.  We can deduce a contradiction as in
     Theorem~\ref{thm-3}. Therefore $g$ is non-constant, i.e. $\mbox{Ent}_w g
     > 0$. Taking the limit, we obtain
      \begin{align}
	\bar{\gamma} \le \frac{\bar{\mathcal{E}}(g, \ln g)}{\mbox{Ent}_w(g)}
	\le \lim_{k \rightarrow + \infty} 
	\frac{\mathcal{E}^{(k)}(f_k, \ln f_k)}{\mbox{Ent}_{\pi^{(k)}} (f_k)}
	= \lim_{k \rightarrow +\infty} \gamma^{(k)} < \bar{\gamma}\,.
      \end{align}
	This contradiction shows that $\lim\limits_{\epsilon \rightarrow 0}
	\gamma_\epsilon = \bar{\gamma}$\,.
  \end{enumerate}
\end{proof}

\section{Conclusion}
\label{sec-conclusion}
In this paper we consider continuous-time Markov chains on finite state space and
focus on the situation when systems' transitions within clusters are much
faster than transitions among clusters. Several asymptotic results are
obtained concerning Kolmogorov backward equation, Poincar\'e constant,
and (modified) logarithmic Sobolev constants. These results validate the
reduced Markov chain as an approximation of the multiscale Markov chain in the
asymptotic limit.
Especially, when understanding the multiscale Markov chain becomes
infeasible, either due to an extremely large state space or limited
information to identify all transition rates, our results will be
instructive as it suggests that the reduced Markov chain can be
a useful approximation of the original one.

On the other hand, while we assume that there are several subsets of the state space 
such that transitions between them are relatively slow, in applications
it might be the case that these subsets are not known a priori and need to be
identified. How to identify (clustering) the slow subsets is an important problem in
the studies of proteins \cite{Deuflhard200039, msm_gen_valid}, principal component analysis \cite{pca_jolliffe,
pca_review}, climates \cite{math_climate_majda} and
network \cite{Girvan11062002, detect-community-network} et al. Readers
are referred to those literatures for more details.

In future work, it might be interesting to consider asymptotic behaviors 
of other constants in \cite{ricci_entropy, curvature_yau}. As more and more real data become
available nowadays, it is also interesting to quantify the approximation error
of the reduced Markov chain using a data-based approach. 
\section*{Acknowledgement}
Part of this work was done when the author was visiting School of Mathematics
Sciences at Peking University and institute of natural sciences at Shanghai Jiao Tong University. 
The hospitality during the stay is gratefully acknowledged.

\appendix
\section{Some useful facts}
\label{app-sec-1}
In this section we collect some results related to continuous-time Markov chain.
Let $n > 1$, $Q$ be an $n \times n$ matrix satisfying $Q_{ij} \ge 0$ for $1
\le i\neq j \le n$ and $Q_{ii} = -\sum\limits_{j\neq i} Q_{ij}$, $1 \le i \le
n$. We will call such matrix as transition rate matrix. Clearly, $Q\bm{1} = 0$, where $\bm{1} = (1,1,\cdots,
1)^T \in \mathbb{R}^n$. Define $P(t) = e^{tQ}$, then 
$P(t)_{ij} \ge 0$ for $1 \le
i,j \le n$, $P(t)P(s) = P(t+s)$, for $t, s\ge
0$ and $P(t)\bm{1} = \bm{1}$. Therefore $P(t)$ are stochastic matrices and
satisfy semigroup property. Let $\Omega =\{1,2,\cdots, n\}$ and
$\mathcal{F} = \{f \,|\, f : \Omega \rightarrow \mathbb{R}\}$, which can
be viewed as $\mathbb{R}^n$. For $f \in \mathcal{F}$, denote
$|f|_{\infty} = \max\limits_{i \in \Omega} |f(i)|$. Then $P(t)$ defines a semigroup on
$\mathcal{F}$ with infinitesimal generator $Q$. It also defines a
continuous-time Markov chain $x(t)$ on $\Omega$ such that ${\bf{P}}(x(t) = j \,|\,
x(0) = i) = P(t)_{ij}$, $1\le i,j \le n$. 
Define $f_t = P(t) f \in \mathcal{F}$ for function $f \in \mathcal{F}$, we have 
\begin{align}
f_t(i) = \mathbf{E}\big(f(x(t))~|~x(0) = i\big)
\label{feynman-kac}
\end{align}
and it satisfies the Kolmogorov backward equation
\begin{align}
  \frac{d}{dt} f_t = Q f_t\,,\quad f_0 = f\,. 
  \label{kbe}
\end{align}
From (\ref{feynman-kac}), we know $|f_t|_{\infty} \le |f_0|_{\infty}$.

Assume the probability distribution of $x(t)$ at time $t \ge 0$ is $\mu_t$, then 
it is known that ${\mu}_t$ satisfies Kolmogorov forward (Fokker-Planck) equation 
\begin{align}
  \dot{\mu}_t = Q^T {\mu}_t\,
  \label{kfe}
\end{align}
with initial value $\mu_0$. Therefore we have ${\mu}_t = P(t)^T \mu_0$.
A probability measure ${\pi}$ is called the invariant measure of Markov chain
$x_t$ iff $Q^T \pi = 0$.
If we further assume the Markov chain is irreducible, then the invariant
measure is unique.
Also define the $\pi$-weighted inner product on
$\mathcal{F}$ as
\begin{align*}
  \langle f , g \rangle_{{\pi}} = \sum_{i \in \Omega} f(i) g(i) \pi(i)\,,
  \quad f, g \in \mathcal{F}\,,
\end{align*}
and the Dirichlet form $\mathcal{E}$ 
\begin{align*}
  \mathcal{E}(f,g) = -\langle f, Qg\rangle_{\pi}\,.
\end{align*}
Let $Q^*$ be the adjoint matrix under $\langle,\rangle_{\pi}$, we have $Q^*_{ij} = 
\frac{Q_{ji}\pi(j)}{\pi(i)}$. 
We can check that $Q^*$ is also a transition rate matrix and $(Q^*)^T\pi =
0$. The corresponding Markov chain defined by $Q^*$ is called the time-reversed Markov
chain. For $f \in \mathcal{F}$, we have 
\begin{align}
  \begin{split}
  \mathcal{E}(f,f) =& -\langle\frac{Q+Q^*}{2}f,f\rangle_{\pi} \\
		   =& \frac{1}{2}\sum_{i,j \in \Omega} \frac{Q_{ij} + Q^*_{ij}}{2} (f(i)
  - f(j))^2 \pi(i)  \\
     =& \frac{1}{2}\sum_{i,j \in \Omega} Q_{ij} (f(i) - f(j))^2 \pi(i) \ge
  0\,.
\end{split}
  \label{dirichlet-nonsym}
\end{align}

Define matrix
$$\Pi = \mbox{diag}\{\pi(1), \pi(2), \cdots, \pi(n)\},$$ then we have the matrix
equation $Q^* = \Pi^{-1}Q^T\Pi$.
It is direct to see that 
\begin{align*}
  & \quad \mbox{Dirichlet form}~\mathcal{E}~ \mbox{is symmetric } \\
  \Longleftrightarrow & \quad Q = Q^* \\
\Longleftrightarrow & \quad \pi(i)Q_{ij} = \pi(j)Q_{ji} \,, \forall i,j
\in \Omega \,.
\end{align*}
In this case, we say $\pi$ satisfies the 
detailed balance condition and the Markov chain is reversible. 

For $\mu_t$ satisfying (\ref{kfe}), we define $\mu_t = \rho_t \pi$, i.e.
$\mu_t(i) = \rho_t(i) \pi(i)$, $i \in \Omega$, then 
\begin{align}
  \frac{d}{dt}\rho_t = Q^* \rho_t\,.
  \label{adjoint-eqn}
\end{align}
When $Q^*=Q$, i.e. the detailed balance condition holds, equation
(\ref{adjoint-eqn}) coincides with the Kolmogorov backward equation
(\ref{kbe}).

Consider the singular value decomposition (SVD) $Q = UDV^T$, where $U, V$ are
$n \times n$ orthogonal matrix. $D = \mbox{diag}\{\sigma_1, \sigma_2, \cdots,
\sigma_n\}$ is a diagonal matrix consisting of singular values $\sigma_1 \ge \sigma_2 \ge \cdots \ge \sigma_n\ge 0$.
Denote $i$th column of matrix $U,V$ as $U_i$, $V_i$, respectively, i.e.  
$U=[U_1, U_2, \cdots, U_n], V=[V_1, V_2, \cdots, V_n]$. Then $\{U_i\}_{1 \le i
\le n}$ and $\{V_i\}_{1 \le i \le n}$ are two orthonormal basis of
$\mathbb{R}^n$. 
Since $Q\bm{1} = 0$, $Q^T\pi = 0$ and using the fact that the invariant measure $\pi$ is unique, we know $\sigma_n = 0 < \sigma_{n-1}$.
We can further deduce that $V_n \parallel \bm{1}$ and $U_n \parallel \pi$.
The linear system $Q^T x = b$ can be studied based on the SVD decomposition. We have 
\begin{lemma}
  Consider linear system $Q^T x=b$, where $x, b \in \mathbb{R}^n$.
  \begin{enumerate}
    \item
      There is a solution if and only if $b^T \bm{1} = 0$.
    \item
      Assume $b^T \bm{1} = 0$, then the solutions can be written as 
      \begin{align}
	x = a \pi + \sum_{i = 1}^{n-1} \sigma_i^{-1} (V_i^Tb) U_i\,,
	\label{qxb-sol}
      \end{align}
      for $\forall a \in \mathbb{R}$. Furthermore,
      \begin{align}
	|x - a\pi|_\infty \le |x - a\pi|_2 \le \sigma_{n-1}^{-1} |b|_2
	\,.
	\label{qxb-bound}
      \end{align}
\end{enumerate}
\label{lem-1}
\end{lemma}
\begin{proof}
  \begin{enumerate}
    \item
      Assume $Q^T x=b$, we have $b^T\bm{1} = x^TQ\bm{1} = 0$. The sufficiency
      follows from the second conclusion.
    \item
      We directly verify that expression (\ref{qxb-sol}) satisfies the equation
      $Q^Tx = b$. $\forall a \in \mathbb{R}$, using $Q^T\pi = 0$, $Q = UDV^T$
      and $U_i$, $V_i$ are orthonormal, we have 
      \begin{align*}
	Q^Tx =& Q^T\Big(a\pi + \sum_{i = 1}^{n-1} \sigma_i^{-1} (V_i^Tb) U_i\Big)  \\
	=& \sum_{i = 1}^{n-1} \sigma_i^{-1} (V_i^Tb)VDU^T U_i \\
	=& \sum_{i = 1}^{n-1} (V_i^Tb)V_i =b\,.
      \end{align*}
      In the last equality, we have used $b^T V_n = b^T \bm{1} = 0$.

      On the contrary, suppose $x \in \mathbb{R}^n$ satisfies the equation $Q^Tx = b$. Since $U_i$ is
      orthonormal and $U_n \parallel \pi$, we can assume $x = a \pi + \sum\limits_{i=1}^{n-1} a_i U_i$.
      Substituting it into $Q^Tx=b$, we obtain $a_i = \sigma_i^{-1} (V_i^Tb)$, $1 \le i \le n-1$, 
      i.e. $x$ is given by (\ref{qxb-sol}).

      And we can estimate
      \begin{align}
      |x - a\pi|_2 = \Big|\sum_{i = 1}^{n-1} \sigma_i^{-1} (V_i^Tb) U_i\Big|_2
      = \Big(\sum_{i=1}^{n-1} \sigma_i^{-2} (V_i^Tb)^2\Big)^{\frac{1}{2}}
      \le \sigma_{n-1}^{-1} \Big(\sum_{i=1}^{n-1} (V_i^Tb)^2\Big)^{\frac{1}{2}}
     =\sigma_{n-1}^{-1} |b|_2\,,
   \end{align}
   i.e. (\ref{qxb-bound}) is proved.
  \end{enumerate}
\end{proof}

In the reversible case, we have $Q=Q^*$ and $\Pi
Q = Q^T\Pi$, which indicates that $\Pi^{\frac{1}{2}}Q\Pi^{-\frac{1}{2}}$ is
symmetric. Consider the eigenvalue decomposition
$\Pi^{\frac{1}{2}}Q\Pi^{-\frac{1}{2}} = UDU^{T}$ with $U^TU=UU^T=I$, $D
= \mbox{diag}\{\lambda_1, \lambda_2, \cdots, \lambda_{n}\}$ is a diagonal matrix
consisting of real eigenvalues. 
Then $Q=TDT^{-1}$, with $T= \Pi^{-\frac{1}{2}}U$. Denote $i$th column vector
of $T$ as $T_i$, i.e. $T= [T_1, T_2, \cdots, T_n]$,
then we have $QT_i = \lambda_i T_i, \langle T_i, T_j\rangle_{\pi} =
\delta_{ij}$. Therefore $T_i$ is the eigenvector of
$Q$ corresponding to eigenvalue $\lambda_i$ and $\{T_i\}_{1 \le i \le n}$
forms an orthonormal basis of $L^2_{\pi}(\Omega)$. 
For two functions $f,g$ on $\Omega$ written as $f = \sum\limits_{i=1}^n f_i T_i$, $g
= \sum\limits_{i=1}^n g_i T_i$, we have 
\begin{align*}
  \mathcal{E}(f,g) = -\langle f, Qg\rangle_\pi = -\sum_{i=1}^n \lambda_i f_i g_i \,.
\end{align*}
From (\ref{dirichlet-nonsym}), we could assume
$0 = \lambda_1 > \lambda_2 \ge \cdots \ge \lambda_{n}$. It is clear that the
Poincar\'e constant $\lambda = -\lambda_2 > 0$.
\section{Asymptotic expansion method : formal argument}
\label{app-sec-2}
Asymptotic expansion method has been widely 
used in studying dynamical systems, partial differential equations in certain limiting regime, see
\cite{nssa_weinan, Zhang_numericalstudy, papanicolau1978asymptotic, pavliotis2008multiscale}. 
In this section, we consider the Kolmogorov backward equation and various
constants studied in Section~\ref{sec-reversible}-\ref{sec-general} using this
method. First consider the equation 
\begin{align}
  \frac{d}{dt} \rho_t = & Q\rho_t = \big(\frac{1}{\epsilon} Q_0 + Q_1\big)
  \rho_t\,. \label{fp-ms-formal} 
\end{align}
 Assume we have the expansion
\begin{align}
\rho_t = \rho_{t,0} + \epsilon\rho_{t,1} + \epsilon^2 \rho_{t,2} + \cdots\,.
\label{rho-expan}
\end{align}
Substitute it into (\ref{fp-ms-formal}), we obtain
\begin{align}
  \frac{d\rho_{t,0}}{dt} 
  + \epsilon \frac{d\rho_{t,1}}{dt} 
  + \mathcal{O}(\epsilon^2)  = 
  \frac{Q_0 \rho_{t,0}}{\epsilon} + Q_0 \rho_{t,1} + Q_1\rho_{t,0} + \epsilon
  Q_0 \rho_{t,2} + \epsilon Q_1 \rho_{t,1} + \mathcal{O}(\epsilon^2)\,.
\end{align}
Collecting terms of order $\mathcal{O}(\frac{1}{\epsilon})$ and $\mathcal{O}(1)$ 
with respect to parameter $\epsilon$,
we obtain equations
\begin{align}
  \begin{split}
  Q_0\rho_{t,0}  =& \, 0 \,, \\
  \frac{d\rho_{t,0}}{dt} =&\, Q_0\rho_{t,1} + Q_1\rho_{t,0} \,.
  \end{split}
  \label{eps-terms}
\end{align}
Since $Q_0$ is a block diagonal matrix of form (\ref{q0-block}), the first
equation in (\ref{eps-terms}) can be written as $Q_{0,i}\, \rho_{t,0}(i,\cdot) =
0$, $1 \le i \le m$.
It follows from the irreducibility of each Markov chain $\mathcal{C}_i$ 
that $\rho_{t,0}$ is constant on each subset $X_i$. And we can assume $\rho_{t,0}(x) =
\bar{\rho}_t(i)$ for $x\in X_i$, where function $\bar{\rho}_t : \bar{X} \rightarrow
\mathbb{R}$. Then the second equation of (\ref{eps-terms}) can be written more explicitly as 
\begin{align}
  \frac{d\bar{\rho}_t}{dt} =& \sum_{x' \in X_i} Q_{0,i}(x,x') \rho_{t,1}(x')
   + \sum_{j \neq i} \sum_{y \in X_j} Q_1(x,y)
(\bar{\rho}_t(j) - \bar{\rho}_t(i)) \,,
\end{align}
where $1 \le i \le m$ and $x \in X_i$. Now we multiply both sides of the above
equation by $\pi_i(x)$ and sum up $x \in X_i$. Using $\pi_i^TQ_{0,i} = 0$ and
the definition $\bar{Q}$ in (\ref{q-bar}), we arrive at 
\begin{align}
  \frac{d\bar{\rho}_t}{dt} = \bar{Q} \bar{\rho}_t\,.
  \label{fp-bar-formal}
\end{align}
From expansion (\ref{rho-expan}), the above reasoning indicates the convergence of $\rho_t$ in
(\ref{fp-ms-formal}) to $\bar{\rho}_t$ in (\ref{fp-bar-formal}). See
Theorem~\ref{thm-0} in Section~\ref{sec-reversible}. 

The asymptotic behavior of Poincar\'e constant $\lambda_\epsilon$, 
logarithmic Sobolev constant $\alpha_\epsilon$ and modified logarithmic Sobolev constant
$\gamma_\epsilon$ can be studied as well.
Let $f^\epsilon$ be a function where the infimum in (\ref{poincare-const})
is achieved. Then standard variation method implies 
\begin{align}
  -\frac{Q + Q^*}{2} f^\epsilon = \lambda_\epsilon f^\epsilon\,,
\end{align}
with $|f^\epsilon|_{L^2(\pi^\epsilon)} = 1$.
Similarly, the minimizer of (\ref{log-sobolev-const}) satisfies 
\begin{align}
  -\frac{Q + Q^*}{2} f^\epsilon = \alpha_\epsilon f^\epsilon \ln
  (f^\epsilon)^2\,,
\end{align}
with $|f^\epsilon|_{L^2(\pi^\epsilon)} = 1$, while the minimizer of (\ref{mlsi-def})
satisfies 
\begin{align}
  -Q^*f^\epsilon - f^\epsilon Q\ln f^\epsilon = \gamma_\epsilon f^\epsilon \ln f^\epsilon\,,
  \label{gamma-euler}
\end{align}
with $\mathbf{E}_{\pi^\epsilon} f^\epsilon = 1$.
For simplicity, we only consider the modified logarithmic Sobolev constant
$\gamma_\epsilon$ using (\ref{gamma-euler}), since constants $\lambda_\epsilon$ and
$\alpha_\epsilon$ can be studied in a similar way. Assume we have the
expansion 
\begin{align}
  f^\epsilon = f_0 + \epsilon f_1 + \cdots\,,  \quad 
  \gamma_\epsilon = \gamma_0 + \epsilon \gamma_1 + \cdots\,.
\end{align}
Substituting it into (\ref{gamma-euler}), we have 
\begin{align*}
  &-\frac{Q_0^*f_0}{\epsilon} - Q_1^*f_0 - Q_0^* f_1 - \frac{f_0Q_0\ln f_0}{\epsilon}
  - f_0Q_1 \ln f_0 - f_1Q_0\ln f_0 - f_0Q_0 \Big(\frac{f_1}{f_0}\Big) +
  \mathcal{O}(\epsilon) \\
  =& \gamma_0 f_0 \ln f_0 + \mathcal{O}(\epsilon) \,.
\end{align*}
Collecting terms of order $\mathcal{O}(\frac{1}{\epsilon})$ and $\mathcal{O}(1)$ 
with respect to parameter $\epsilon$, we obtain equations
\begin{align}
  \begin{split}
  &Q_0^*f_0+f_0Q_0\ln f_0 = 0 \,,\\
&- Q_1^*f_0 - Q_0^* f_1  - f_0Q_1 \ln f_0 - f_1Q_0\ln f_0 - f_0Q_0 \Big(\frac{f_1}{f_0}\Big) 
  = \gamma_0 f_0 \ln f_0 \,.
\end{split}
\label{gamma-expan}
\end{align}
Now for each $1 \le i \le m$, we multiply both sides of the first equation of (\ref{gamma-expan}) by
$\pi_i(x)$ and sum up $x \in X_i$. Using $(Q_{0,i}^*)^T\pi_i = 0$, we can obtain
\begin{align*}
\mathcal{E}_i(f_0(i, \cdot), \ln f_0(i, \cdot)) = 0\,,
\end{align*}
where $\mathcal{E}_i$ is the Dirichlet form of Markov chain $\mathcal{C}_i$.
From Lemma~$2.7$ of \cite{log-sob-diaconis}, we know 
\begin{align*}
  \mathcal{E}_i(f_0^{\frac{1}{2}}(i, \cdot),f_0^{\frac{1}{2}}(i, \cdot))
  \le \frac{1}{2} \mathcal{E}_i(f_0(i, \cdot), \ln f_0(i, \cdot)) = 0\,.
\end{align*}
Since Markov chain $\mathcal{C}_i$ is irreducible, we can deduce that $f_0$ is
constant on each subset $X_i$, i.e. we have $f_0(x) = \bar{f}(i)$ when $x \in
X_i$, where $\bar{f} : \bar{X} \rightarrow \mathbb{R}$.  Now we multiply both
sides of the second equation in (\ref{gamma-expan}) by $\pi_i(x)$ and sum up
$x \in X_i$. Using the fact that $Q_0 \ln f_0 = 0$, $Q_{0,i}^T \pi_i =
(Q_{0,i}^*)^T \pi_i = 0$, we can deduce that
\begin{align}
  -\bar{Q}^* \bar{f} - \bar{f} \bar{Q} \ln \bar{f} = \gamma_0 \bar{f} \ln
  \bar{f}\,,
\end{align}
with $\mathbf{E}_w \bar{f} = 1$ (see Theorem~\ref{thm-pi-general}). 
Comparing to (\ref{gamma-euler}), this equation shows that function $\bar{f}$ is a minimizer of 
\begin{align}
  \frac{\bar{\mathcal{E}}(\bar{f}, \ln \bar{f})}{\mbox{Ent}_w \bar{f}}
  \label{mlsi-bar}
\end{align}
and takes value $\gamma_0$. 
Using the fact that $\bar{\gamma}$ is the infimum of (\ref{mlsi-bar}) and 
the fact $\lim\limits_{\epsilon \rightarrow 0} \gamma_\epsilon \le
\bar{\gamma}$ (see Theorem~\ref{thm-constants-general}), we have
\begin{align*}
  \bar{\gamma} \le \gamma_0 = \lim_{\epsilon \rightarrow 0} \gamma_\epsilon
  \le \bar{\gamma}\,.
\end{align*}
Therefore we conclude that $\lim\limits_{\epsilon \rightarrow 0} \gamma_\epsilon =
\bar{\gamma}$.

\bibliographystyle{siam}
\bibliography{reference}
\end{document}